\documentclass[11pt]{article}
\usepackage[noblocks]{authblk}

\usepackage[T1]{fontenc}
\usepackage{amsmath,amssymb,amsfonts,amsthm}
\usepackage{mathrsfs}
\usepackage{graphicx}
\usepackage{booktabs}
\usepackage{cite}
\usepackage[hidelinks]{hyperref}

\numberwithin{equation}{section}

\newtheorem{theorem}{Theorem}[section]
\newtheorem{lemma}{Lemma}[section]
\newtheorem{proposition}{Proposition}[section]
\newtheorem{corollary}{Corollary}[section]

\theoremstyle{remark}
\newtheorem{remark}{Remark}[section]

\title{A nonlocal nonlinear Schr{\"o}dinger model: well-posedness, local limit, and structure-preserving asymptotically compatible Fourier approximations}

\author[1]{Jiashu Lu%
	\thanks{Corresponding author. E-mail address:
		\texttt{jslu@xaut.edu.cn}.}}

\author[2]{Yufeng Nie}
\author[2]{Xinning Xie}
\author[3]{Pingrui Zhang}

\affil[1]{School of Mathematics,
	Xi'an University of Technology,
	Xi'an 710054, China}

\affil[2]{School of Mathematics and Statistics,
	Northwestern Polytechnical University,
	Xi'an 710129, China}

\affil[3]{Shaoxing Institute,
	Zhejiang University,
	Shaoxing 312099, China}

\date{}

\date{}

\begin{document}
	
	\maketitle
	
	\begin{abstract}
		This paper studies a finite-horizon nonlocal nonlinear Schr{\"o}dinger (NLS) model on periodic domains and develops a structure-preserving asymptotically compatible Fourier collocation (pseudo-spectral) method. We establish global well-posedness and conservation of mass and nonlocal energy for each fixed horizon, prove an $O(\delta^2)$ local limit under suitable regularity assumptions, and characterize finite-horizon dispersion at low and high frequencies. The method combines Crank--Nicolson time stepping with Fourier collocation and exactly preserves the grid mass and energy. The existence, uniqueness, and convergence of the numerical solutions are proved. In particular, the error of the nonlocal NLS solution is $O(\tau^2+N^{s-r})$  in the $H^s$ norm, uniformly with respect to the horizon. Moreover, its total $H^s$-error relative to the local NLS solution is $O(\delta^2+\tau^2+N^{s-r})$ without any coupling condition among $\delta$, $\tau$, and $N$, which proves asymptotic compatibility of the proposed method.  Numerical experiments in one, two, and three dimensions are presented to verify the theoretical accuracy and discrete conservation, confirm convergence under independent variation of horizon and discretization parameters, and show how the horizon and kernel affect dispersive wave propagation.
	\end{abstract}
	
	\medskip
	
	
	\smallskip
	

\section{Introduction}

This paper studies a finite-horizon nonlocal nonlinear Schr{\"o}dinger (NLS) model, including its well-posedness, conservation laws, local limit, and dispersion properties, together with structure-preserving asymptotically compatible Fourier approximations. The model equation is
\begin{equation}
	\mathrm i\,\partial_tu_\delta=-\mathcal L_\delta u_\delta+f(|u_\delta|^2)u_\delta,
\end{equation}
where $\mathcal L_\delta$ is a nonlocal diffusion operator with finite horizon introduced in \cite{2012Analysis}. Such operators describe spatial interactions through integration rather than differentiation, which allows long range interactions and spatial discontinuities to be treated naturally and has led to applications in material damage modeling\cite{javili2019peridynamics,shen2021peridynamic,lipton2025energy}, traffic flow modeling\cite{huang2022stability,huang2024asymptotic}, neural network optimization\cite{tao2018nonlocal,you2022nonlocal}, biological modeling\cite{pal2025nonlocal}, and related areas. For the present model, $\mathcal L_\delta$ integrates over a spherical neighborhood of radius $\delta$, where $\delta>0$ is the horizon and determines the interaction range. Under the kernel assumptions specified in section \ref{sec2}, $\mathcal L_\delta$ can converge to the local Laplacian as $\delta\to0$, and the nonlocal NLS model can recover the classical NLS equation. 

An typical example of an NLS equation with nonlocal linear dispersion is the fractional NLS, which has been studies in both theory and applications. The fractional dispersive operator arises naturally from the L{\'e}vy path-integral formulation of fractional quantum mechanics, which leads to the fractional Schr{\"o}dinger equation \cite{laskin2000fractional}. In addition, Fractional NLS equations have also been established for nonlinear lattice systems with long-range interactions\cite{kirkpatrick2013continuum}. Moreover, the fractional Laplacian have an infinite interaction range and can be recovered from nonlocal diffusion operators with $\delta\to\infty$ \cite{d2013fractional}. Thus, the nonlocal NLS framework considered here may provide a bridge between classical and fractional NLS equations.

The finite-horizon nonlocal Schr{\"o}dinger equations have previously been studied numerically on the real line. Exact artificial nonreflecting boundary conditions were constructed in \cite{yan2020numerical}, and a second-order Crank--Nicolson approximation with DtN-type artificial boundary conditions was subsequently analyzed in \cite{wang2022stability}. These works concern one-dimensional linear equations on unbounded domains and focus primarily on reducing the resulting infinite systems to finite computational domains. Therefore, the multidimensional periodic nonlocal NLS considered here therefore requires a different analysis framework.

Modified Crank--Nicolson methods that conserve discrete mass and energy were developed for the local NLS equation in \cite{sanz1984methods,henning2017crank}, and related structure-preserving methods have been developed for fractional and coupled NLS equations \cite{wang2018structure,ding2024construction,zhang2025high}. On a periodic domain, the translation invariant nonlocal operator is diagonal in the Fourier basis, which makes a Fourier method natural for both analysis and computation \cite{2017Fast,alali2020fourier}. Moreover, since the nonlocal NLS model can converge the classical NLS model as $\delta\to0$, its numerical approximation should preserve the same limit, i.e., the numerical solution can converge to the local NLS solution when $\delta\to0$ and the temporal and spatial discretizations are refined independently, which is known as asymptotic compatibility and depends on the choice of discretization \cite{tian2014asymptotically,tian2020asymptotically}. In this paper, we develop a Crank--Nicolson Fourier collocation (pseudo-spectral) scheme with a symmetric difference quotient of the nonlinear potential and a self-adjoint Fourier collocation discretization of the nonlocal operator, and establish the structure-preserving asymptotically compatible framework of the proposed method.

The main results of this paper are summarized as follows.
\begin{enumerate}
	\item[(i)] We establish a continuous framework for the finite-horizon nonlocal NLS, including global well-posedness and conservation of mass and nonlocal energy for each fixed horizon, the $O(\delta^2)$ local limit under suitable regularity assumptions, and the low- and high-frequency dispersion properties.
	\item[(ii)] We construct an implementable Crank--Nicolson Fourier collocation method that exactly preserves the grid mass and the discrete counterpart of the original nonlocal energy.
	\item[(iii)] We prove the horizon-uniform $H^s$-error estimate $O(\tau^2+N^{s-r})$ for the nonlocal solution, where $\tau$ is the time step, $N$ is the Fourier cutoff, and the exact solution has the assumed $H^r$ regularity.
	\item[(iv)] Combining the fully discrete estimate with the continuous local limit gives the total error $O(\delta^2+\tau^2+N^{s-r})$ relative to the local NLS solution, without any coupling condition among $\delta$, $\tau$, and $N$, and hence proves asymptotic compatibility.
	\item[(v)] Numerical experiments in one, two, and three dimensions confirm the convergence rates, conservation laws, and asymptotic compatibility, while a two-dimensional nonlinear wave-packet interaction illustrates the finite-horizon dispersive effects.
\end{enumerate}

The remainder of the paper is organized as follows. Section \ref{sec2} defines the nonlocal NLS model and its Fourier representation. Section \ref{sec3} introduces the Crank--Nicolson, Fourier--Galerkin, and Fourier collocation discretizations. Section \ref{sec4} proves well-posedness, conservation laws, the nonlocal to local limit, and the dispersion properties of the model. Section \ref{sec5} analyzes the time discretization. Section \ref{sec6} establishes the Fourier approximation estimates, the collocation error bound, and asymptotic compatibility. Numerical results are reported in section \ref{sec7}, followed by conclusions in section \ref{sec8}.
\section{The nonlocal NLS model}\label{sec2}
This section introduces the nonlocal NLS model considered in this paper. We first define the periodic nonlocal diffusion operator and its Fourier multiplier, and then present the nonlinear evolution equation together with its local limit and conservative structure. The assumptions stated in this section remain in force throughout the paper.
\subsection{Nonlocal operator and Fourier representation}
First, we let $\mathbb T^d=(\mathbb R/2\pi\mathbb Z)^d$ with $d\in\{1,2,3\}$, and we use $[-\pi,\pi)^d$ as a representative fundamental cell of $\mathbb T^d$. Every function $v:\mathbb T^d\to\mathbb C$ is identified with its $2\pi$-periodic extension to $\mathbb R^d$, still denoted by $v$, satisfying
\begin{equation}
	v(x+2\pi m)=v(x),\quad x\in\mathbb R^d,\quad m\in\mathbb Z^d.
\end{equation}

For complex-valued functions $u$ and $v$, we use the normalized inner product
\begin{equation}
	(u,v)=\frac{1}{(2\pi)^d}
	\int_{\mathbb T^d}u(x)\overline{v(x)}\,\mathrm dx,
\end{equation}
with the induced norm $\|v\|_{L^2}^2=(v,v)$. Moreover, the Fourier basis is denoted by
\begin{equation}
	e_k(x)=\exp(\mathrm i k\cdot x),\quad k\in\mathbb Z^d.
\end{equation}
Above $\{e_k\}_{k\in\mathbb Z^d}$ forms an orthonormal basis of $L^2(\mathbb T^d)$. Hence, every $v\in L^2(\mathbb T^d)$ has the Fourier representation:
\begin{equation}
	v(x)=\sum_{k\in\mathbb Z^d}\widehat v_k e_k(x),\quad\widehat v_k=(v,e_k),
\end{equation}
where the series converges to $v$ in $L^2(\mathbb T^d)$.  Moreover, the periodic Sobolev space $H^q(\mathbb T^d)$ is characterized by the norm
\begin{equation}\label{Hnorm}
	\|v\|_{H^q}^2=\sum_{k\in\mathbb Z^d}(1+|k|^2)^q|\widehat v_k|^2.
\end{equation}

Let $B_\delta=\{z\in\mathbb R^d:|z|<\delta\}$, where the horizon parameter satisfies $0<\delta\le \overline\delta<\pi$. In this paper, we consider kernels of the form
\begin{equation}
	\rho_\delta(z)=\delta^{-d-2}\rho\!\left(\frac{z}{\delta}\right).
\end{equation}
where the reference kernel $\rho$ satisfies
\begin{equation}\label{K1}
	\rho\in L^1(B_1),
	\quad
	\rho(z)=\rho(-z)\ge0,
	\quad
	\operatorname{supp}\rho\subset \overline{B_1},
\end{equation}
and the second-moment normalization
\begin{equation}\label{K2}
	\frac12\int_{B_1}\rho(z)z_i z_j\,\mathrm dz=\delta_{ij},
	\quad 1\le i,j\le d.
\end{equation}
Following the nonlocal diffusion framework developed in \cite{2012Analysis,2013A}, we define the nonlocal diffusion operator $\mathcal L_\delta$ acting on a periodic function $v$ by
\begin{equation}\label{2.1}
	\mathcal L_\delta v(x)=\int_{B_\delta}\rho_\delta(z)\bigl(v(x+z)-v(x)\bigr)\,\mathrm dz,
\end{equation}
where $x+z$ is understood modulo $2\pi$ in each coordinate. Since $\rho_\delta\in L^1(B_\delta)$, the operator $\mathcal L_\delta$ is bounded on $L^2(\mathbb T^d)$, and more generally on every periodic Sobolev space $H^q(\mathbb T^d)$, for each fixed $\delta>0$.

Moreover, the operator \(-\mathcal L_\delta\) induces the Hermitian sesquilinear form
\begin{equation}\label{2.2}
	\begin{aligned}
		a_\delta(u,v)=\frac{1}{2(2\pi)^d}\int_{\mathbb T^d}\int_{B_\delta}
		&\rho_\delta(z)\bigl(u(x+z)-u(x)\bigr)\\
		&\times\overline{\bigl(v(x+z)-v(x)\bigr)}\,\mathrm dz\,\mathrm dx.
	\end{aligned}
\end{equation}
The symmetry of $\rho_\delta$, together with periodic changes of variables, gives the nonlocal integration-by-parts identity\cite{2012Analysis,2013A}
\begin{equation}\label{integralbyparts}
	(-\mathcal L_\delta u,v)=a_\delta(u,v).
\end{equation}
Moreover, since $\rho_\delta\ge0$, we have
\begin{equation}
	a_\delta(u,u)\ge0.
\end{equation}
The above relations show that \(-\mathcal L_\delta\) is self-adjoint and nonnegative. 

Moreover, since  $\mathcal L_\delta$ is translation invariant, each Fourier mode $e_k$ is an eigenfunction of $-\mathcal L_\delta$. A direct calculation leads to
\begin{equation}\label{2.3}
	-\mathcal L_\delta e_k=\lambda_\delta(k)e_k,
	\quad k\in\mathbb Z^d,
\end{equation}
where the corresponding eigenvalue is
\begin{equation}\label{2.4}
	\lambda_\delta(k)=\int_{B_\delta}\rho_\delta(z)\bigl(1-\cos(k\cdot z)\bigr)\,\mathrm dz.
\end{equation}
In particular, we have $\lambda_\delta(0)=0$ and $\lambda_\delta(k)=\lambda_\delta(-k)\ge0$. \eqref{2.3} shows that $-\mathcal L_\delta$ is a Fourier multiplier with the nonnegative symbol $\lambda_\delta(k)$. Consequently,
\begin{equation}
	\widehat{(-\mathcal L_\delta v)}_k
	=\lambda_\delta(k)\widehat v_k.
\end{equation}
Parseval's identity further gives
\begin{equation}\label{avvpar}
	a_\delta(u,v)=\sum_{k\in\mathbb Z^d}\lambda_\delta(k)\widehat u_k\overline{\widehat v_k},
	\quad
	a_\delta(v,v)=\sum_{k\in\mathbb Z^d}\lambda_\delta(k)|\widehat v_k|^2.
\end{equation}

\subsection{Nonlocal NLS model}
Let $f:[0,\infty)\to\mathbb R$ be a real-valued function satisfying 
\begin{equation}\label{N1}
	f\in C_{\mathrm{loc}}^{r+3}([0,\infty)),
\end{equation}
with $r\in\mathbb N_0$. We use the normalized primitive $F(\eta)=\int_0^\eta f(\xi)\,\mathrm d\xi$, so that $F(0)=0$. A typical example is
\begin{equation}
	f(\eta)=\beta\eta^p,\quad F(\eta)=\frac{\beta}{p+1}\eta^{p+1}.
	\quad\beta\in\mathbb R,\quad
	p\in\mathbb N,
\end{equation}

Then, we consider the periodic finite-horizon nonlocal NLS equation
\begin{equation}\label{2.5}
	\begin{cases}
		\mathrm i\,\partial_tu_\delta=-\mathcal L_\delta u_\delta+f(|u_\delta|^2)u_\delta,&x\in\mathbb T^d,\quad 0<t\le T,\\
		u_\delta(x,0)=u_{\delta,0}(x),&x\in\mathbb T^d,
	\end{cases}
\end{equation}
where $u_{\delta,0}$ is a prescribed complex-valued periodic initial function. Its precise regularity will be specified in the well-posedness and error estimates.

The corresponding local NLS equation is
\begin{equation}\label{2.6}
	\begin{cases}
		\mathrm i\,\partial_tu=-\Delta u+f(|u|^2)u,&x\in\mathbb T^d,\quad 0<t\le T,\\
		u(x,0)=u_0(x),&x\in\mathbb T^d.
	\end{cases}
\end{equation}
The local initial function $u_0$ is also periodic. We assume that \eqref{2.6} admits a sufficiently regular solution on the time interval $[0,T]$ considered below. And the convergence of the solution of \eqref{2.5} to that of \eqref{2.6} as $\delta\to0$ will be established in section \ref{sec4.3}. Moreover, replacing the Laplacian by the nonlocal diffusion operator changes the linear frequency associated with the Fourier mode $e_k$ from $|k|^2$ to $\lambda_\delta(k)$. The resulting dispersion relation, group velocity, and high-frequency behavior are studied in section \ref{sec4.4}.

To describe the conservative structure of the nonlocal NLS model, we introduce the mass and energy functionals associated with \eqref{2.5}. For a periodic function $v$, its mass is defined by
\begin{equation}\label{2.7}
	M(v)=\|v\|_{L^2}^2.
\end{equation}
And for $v\in H^s(\mathbb T^d)$, the corresponding nonlocal energy is
\begin{equation}\label{2.8}
	E_\delta(v)=\frac12a_\delta(v,v)+\frac{1}{2(2\pi)^d}\int_{\mathbb T^d}F(|v(x)|^2)\,\mathrm dx,
\end{equation}
while the corresponding local energy is
\begin{equation}\label{2.9}
	E_0(v)=\frac12\|\nabla v\|_{L^2}^2+\frac{1}{2(2\pi)^d}\int_{\mathbb T^d}F(|v(x)|^2)\,\mathrm dx.
\end{equation}
The conservation of $M$ and $E_\delta$ with the solutions of \eqref{2.5} will be proved in section \ref{sec4.2}.

\section{Crank--Nicolson Fourier discretizations}\label{sec3}
This section introduces the temporal and spatial discretizations of \eqref{2.5}. We first present a Crank--Nicolson time discretization based on a symmetric difference quotient of the nonlinear potential. We then introduce the Fourier collocation scheme used in computation and an auxiliary Fourier--Galerkin approximation used in its convergence analysis. The two formulations are stated separately because the Galerkin energy is based on the continuous inner product, whereas the collocation energy is based on the grid inner product.
\subsection{Crank--Nicolson time discretization}
Let $N_T$ be a positive integer, set $\tau=T/N_T$, and define $t_n=n\tau$ for $0\le n\le N_T$. Then for a sequence $\{U^n\}_{n=0}^{N_T}$, we denote
\begin{equation}\label{3.1}
	D_\tau U^n=\frac{U^{n+1}-U^n}{\tau},\quad U^{n+\frac12}=\frac{U^{n+1}+U^n}{2}.
\end{equation}
And for $a,b\ge0$, define the symmetric difference quotient
\begin{equation}\label{3.2}
	G(a,b)=
	\begin{cases}
		\displaystyle\frac{F(a)-F(b)}{a-b},&a\ne b,\\
		f(a),&a=b.
	\end{cases}
\end{equation}
Then we have
\begin{equation}\label{3.3}
	G(a,b)=\int_0^1 f\bigl((1-\theta)b+\theta a\bigr)\,\mathrm d\theta.
\end{equation}
For complex-valued functions $v$ and $w$, we further set
\begin{equation}\label{3.5}
	\mathscr G(v,w)=G(|v|^2,|w|^2)\frac{v+w}{2}.
\end{equation}
Then the time-discrete approximation of \eqref{2.5} is: find $U_\delta^{n+1}$ such that
\begin{equation}\label{3.6}
	\mathrm iD_\tau U_\delta^n=-\mathcal L_\delta U_\delta^{n+\frac12}+\mathscr G(U_\delta^{n+1},U_\delta^n),\quad 0\le n\le N_T-1,
\end{equation}
with $U_\delta^0=u_{\delta,0}$. The quotient \eqref{3.2} is a modified Crank--Nicolson treatment of a general NLS nonlinearity. The solvability and conservation properties of \eqref{3.6} will be established in section \ref{sec5}.
\subsection{Fourier--Galerkin discretization}
The Fourier--Galerkin discretization is introduced as an auxiliary problem for the subsequent error analysis, rather than as the scheme used in computation. It separates the Fourier projection error from the interpolation and aliasing errors arising in the collocation scheme.

For $N\in\mathbb N$, we define
\begin{equation}\label{3.7}
	K_N=\bigl\{k=(k_1,\ldots,k_d)\in\mathbb Z^d:|k_\ell|\le N,\ 1\le \ell\le d\bigr\}
\end{equation}
and the trigonometric polynomial space
\begin{equation}\label{3.8}
	X_N=\operatorname{span}\{e_k:k\in K_N\}.
\end{equation}
For \(v_N\in X_N\), the operator \(-\mathcal L_\delta\) acts diagonally on its Fourier expansion:
\begin{equation}\label{3.10}
	\widehat{(-\mathcal L_\delta v_N)}_k
	=\lambda_\delta(k)\widehat v_{N,k},
	\qquad k\in K_N,
	\quad v_N\in X_N.
\end{equation}
Let $P_N:L^2(\mathbb T^d)\to X_N$ denote the $L^2$-orthogonal projection,
\begin{equation}\label{3.9}
	P_Nv=\sum_{k\in K_N}\widehat v_k e_k.
\end{equation}
Then the fully discrete Fourier--Galerkin scheme is: find $U_{\delta,N}^{n+1}\in X_N$ such that
\begin{equation}\label{3.12}
	\mathrm i(D_\tau U_{\delta,N}^n,v_N)=a_\delta(U_{\delta,N}^{n+\frac12},v_N)+(\mathscr G(U_{\delta,N}^{n+1},U_{\delta,N}^n),v_N),\quad\forall v_N\in X_N,
\end{equation}
with $U_{\delta,N}^0=P_Nu_{\delta,0}$. Equivalently,
\begin{equation}\label{3.13}
	\mathrm iD_\tau U_{\delta,N}^n=-\mathcal L_\delta U_{\delta,N}^{n+\frac12}+P_N\mathscr G(U_{\delta,N}^{n+1},U_{\delta,N}^n).
\end{equation}

All inner products in \eqref{3.12} are exact $L^2$ inner products. Since $-\mathcal L_\delta$ maps $X_N$ into itself, no projection is required for the linear term. The nonlinear term, however, does not generally belong to $X_N$ and is therefore replaced by its $L^2$-orthogonal projection. Consequently, \eqref{3.12} is equivalent to \eqref{3.13}. 

The Fourier collocation scheme introduced below replaces the exact projection $P_N$ of the nonlinear term by trigonometric interpolation based on its values at the uniform grid points.

\subsection{Fourier collocation and implementation}\label{sec3.3}
First, let $M=2N+1$, $h=2\pi/M$, and introduce the index set
\begin{equation}\label{3.14}
	\mathbb J_M=\{0,1,\ldots,M-1\}^d
\end{equation}
and the uniform grid
\begin{equation}\label{3.15}
	x_j=-\pi\mathbf 1+h j,\quad j\in\mathbb J_M,
\end{equation}
where $\mathbf 1=(1,\ldots,1)\in\mathbb R^d$. For grid functions $V=(V_j)_{j\in\mathbb J_M}$ and $W=(W_j)_{j\in\mathbb J_M}$, define the normalized trapezoidal inner product and norm by
\begin{equation}\label{3.16}
	\langle V,W\rangle_h=\frac1{M^d}\sum_{j\in\mathbb J_M}V_j\overline{W_j},\quad\|V\|_h^2=\langle V,V\rangle_h.
\end{equation}
For $k\in K_N$, let $e_k^h=(e_k(x_j))_{j\in\mathbb J_M}$ denote the restriction of the Fourier mode $e_k$ to the grid. These grid modes are orthonormal with respect to $\langle\cdot,\cdot\rangle_h$. Hence, the discrete Fourier transform of $V$ is the coefficient vector $\widehat V=(\widehat V_k)_{k\in K_N}$, whose components are
\begin{equation}\label{3.17}
	\widehat V_k=\langle V,e_k^h\rangle_h=\frac1{M^d}\sum_{j\in\mathbb J_M}V_j e^{-\mathrm i k\cdot x_j},
	\quad k\in K_N,
\end{equation}
and the inverse transform is
\begin{equation}\label{3.18}
	V_j=\sum_{k\in K_N}\widehat V_k e_k(x_j)=\sum_{k\in K_N}\widehat V_k e^{\mathrm i k\cdot x_j},\quad j\in\mathbb J_M.
\end{equation}

Moreover, the discrete Parseval identity reads
\begin{equation}\label{3.19}
	\langle V,W\rangle_h=\sum_{k\in K_N}\widehat V_k\overline{\widehat W_k},
\end{equation}
and define the Fourier collocation operator $\mathcal L_{\delta,N}^{c}$ by
\begin{equation}\label{3.20}
	(\mathcal L_{\delta,N}^{c}V)_j=-\sum_{k\in K_N}\lambda_\delta(k)\widehat V_k e^{\mathrm i k\cdot x_j}.
\end{equation}
Since $\lambda_\delta(k)$ is real and nonnegative, \eqref{3.19} - \eqref{3.20} give
\begin{equation}
	\langle-\mathcal L_{\delta,N}^{c}V,W\rangle_h=\langle V,-\mathcal L_{\delta,N}^{c}W\rangle_h,\quad\langle-\mathcal L_{\delta,N}^{c}V,V\rangle_h\ge0.
\end{equation}
Thus, $-\mathcal L_{\delta,N}^{c}$ is self-adjoint and nonnegative with respect to the trapezoidal inner product. Replacing the spatial operator in the time-discrete scheme \eqref{3.6} by $-\mathcal L_{\delta,N}^{c}$ and evaluating the nonlinear term pointwise at the grid points yields the Fourier collocation scheme: given $U_j^0=u_{\delta,0}(x_j), j\in\mathbb J_M$, find $U^{n+1}=(U_j^{n+1})_{j\in\mathbb J_M}$ such that
\begin{equation}\label{3.22}
	\mathrm iD_\tau U_j^n=-(\mathcal L_{\delta,N}^{c}U^{n+\frac12})_j+G(|U_j^{n+1}|^2,|U_j^n|^2)U_j^{n+\frac12},\quad j\in\mathbb J_M.
\end{equation}
The associated grid mass and grid energy are
\begin{equation}\label{3.23}
	M_h(V)=\|V\|_h^2
\end{equation}
and
\begin{equation}\label{3.24}
	\begin{aligned}
		E_{\delta,h}(V)=&\frac12\langle-\mathcal L_{\delta,N}^{c}V,V\rangle_h+\frac1{2M^d}\sum_{j\in\mathbb J_M}F(|V_j|^2)\\
		=&\frac12\sum_{k\in K_N}\lambda_\delta(k)|\widehat V_k|^2+\frac1{2M^d}\sum_{j\in\mathbb J_M}F(|V_j|^2).
	\end{aligned}
\end{equation}
The pointwise discrete gradient in \eqref{3.22} and the self-adjoint multiplier \eqref{3.20} yield exact conservation of $M_h$ and $E_{\delta,h}$ when the nonlinear system is solved exactly. This property will be proved later. 

Moreover, one time step of \eqref{3.22} can be computed by fixed-point iteration. Given $U^n$, set $W^{(0)}=U^n$ and solve
\begin{equation}\label{3.25}
	\left(\frac{\mathrm i}{\tau}I+\frac12\mathcal L_{\delta,N}^{c}\right)W^{(m+1)}=\left(\frac{\mathrm i}{\tau}I-\frac12\mathcal L_{\delta,N}^{c}\right)U^n+\mathscr G_h(W^{(m)},U^n),
\end{equation}
where
\begin{equation}\label{3.26}
	[\mathscr G_h(V,W)]_j=G(|V_j|^2,|W_j|^2)\frac{V_j+W_j}{2}.
\end{equation}
Denote its $k$-th discrete Fourier coefficient by
\begin{equation}
	\widehat{\mathscr G}_{h,k}(V,W):=\langle\mathscr G_h(V,W),e_k^h\rangle_h,\quad k\in K_N.
\end{equation}
Then in Fourier space, \eqref{3.25} becomes
\begin{equation}\label{3.27}
	\begin{aligned}
		\left(\frac{\mathrm i}{\tau}-\frac12\lambda_\delta(k)\right)\widehat{W^{(m+1)}}_k=&\left(\frac{\mathrm i}{\tau}+\frac12\lambda_\delta(k)\right)\widehat{U^n}_k\\
		+&\widehat{\mathscr G}_{h,k}(W^{(m)},U^n),
		\qquad k\in K_N.
	\end{aligned}
\end{equation}
The diagonal factors on the left-hand side of \eqref{3.27} are nonzero. Each iteration therefore requires only pointwise operations and a fixed number of FFTs, with computational cost $O(M^d\log M)$ and storage $O(M^d)$.
\begin{remark}
	Scheme \eqref{3.22} differs from the Galerkin scheme \eqref{3.12} in the treatment of the nonlinear term. The former evaluates the nonlinearity pointwise on the grid and preserves the grid energy \(E_{\delta,h}\), and the latter uses the exact \(L^2\)-projection and preserves \(E_\delta\) restricted to \(X_N\). For a general nonpolynomial \(f\), the two schemes are not algebraically identical. In section \ref{sec6}, we first estimate the Galerkin approximation and then controls the interpolation and aliasing errors of the collocation scheme.
\end{remark}
\begin{remark}
	The multipliers $\lambda_\delta(k)$ are independent of time and are therefore precomputed. For small $\delta|k|$, the moment expansion may be used to avoid cancellation in \eqref{2.4}; otherwise, accurate quadrature or the hybrid algorithm in \cite{2017Fast} can be applied. The computed multipliers are real and nonnegative, preserving the self-adjointness and nonnegativity of the discrete operator.
\end{remark}
\section{Analytical properties of the nonlocal NLS model}\label{sec4}
In this section, we establish the analytical properties of the nonlocal NLS model. After establishing the required operator and nonlinear estimates, we prove well-posedness and conservation laws, derive quantitative nonlocal-to-local convergence, and prove the uniform time regularity required by the temporal error analysis. Finally, we analyze the plane-wave dispersion relation and group velocity. Throughout this section, the constants are independent of $0<\delta\le\overline\delta$, unless a dependence on $\delta$ is indicated explicitly.
\subsection{Operator and nonlinear estimates}
Let \(r\in\mathbb N_0\) be the regularity index appearing in \eqref{N1}. For $q\in\mathbb R$, set $\Lambda^q=(I-\Delta)^{q/2}$ and $(v,w)_{H^q}=(\Lambda^qv,\Lambda^qw)$.

\begin{lemma}[Sobolev--Moser estimates]\label{lemmaA}
	Let $q\in\mathbb N_0$ satisfy $q>d/2$, and let $\Phi\in C_{\mathrm{loc}}^{q+3}(\mathbb R^m;\mathbb R^\ell)$. Then the pointwise composition $v\mapsto\Phi\circ v$ defines a twice continuously Fréchet differentiable mapping from $H^q(\mathbb T^d;\mathbb R^m)$ to $H^q(\mathbb T^d;\mathbb R^\ell)$. For every $R>0$, there exists $C_{q,R}>0$, depending on $\Phi$, $q$, and $R$, such that
	\begin{equation}\label{A1}
		\|\Phi\circ v-\Phi\circ w\|_{H^q}
		\le C_{q,R}\|v-w\|_{H^q},
		\quad
		\|v\|_{H^q}+\|w\|_{H^q}\le R,
	\end{equation}
	and, for $j=1,2$,
	\begin{equation}\label{A2}
		\|D^j\Phi(v)[w_1,\ldots,w_j]\|_{H^q}
		\le C_{q,R}\prod_{\nu=1}^j\|w_\nu\|_{H^q},
		\quad
		\|v\|_{H^q}\le R.
	\end{equation}
	Moreover, if $\Phi(0)=0$, then
	\begin{equation}\label{A3}
		\|\Phi\circ v\|_{H^q}
		\le C_q(\|v\|_{L^\infty})\|v\|_{H^q},
	\end{equation}
	where $C_q:[0,\infty)\to(0,\infty)$ depends on $\Phi$ and can be chosen nondecreasing.
\end{lemma}
\begin{proof}
	Since $q>d/2$, $H^q(\mathbb T^d)$ is embedded in $L^\infty(\mathbb T^d)$ and is a Banach algebra. Then the differentiability and estimates \eqref{A1} - \eqref{A3} follow from the Sobolev--Moser product and composition estimates \cite[Chapter 13, Propositions 3.7 and 3.9]{taylor2013partial}. 
\end{proof}

\begin{lemma}\label{lemma4.1}
	Assume \eqref{K1} and \eqref{K2} hold. Then, for every $k\in\mathbb Z^d$, we have
	\begin{equation}\label{4.1}
		0\le \lambda_\delta(k)\le\min\left\{|k|^2,2\delta^{-2}\|\rho\|_{L^1(B_1)}\right\}.
	\end{equation}
	Consequently, for every $q\in\mathbb R$,
	\begin{equation}\label{4.2}
		\|\mathcal L_\delta v\|_{H^q}\le \|v\|_{H^{q+2}},\quad v\in H^{q+2}(\mathbb T^d),
	\end{equation}
	uniformly for $0<\delta\le\overline\delta$. And for each fixed $\delta>0$,
	\begin{equation}\label{4.3}
		\|\mathcal L_\delta v\|_{H^q}\le2\delta^{-2}\|\rho\|_{L^1(B_1)}\|v\|_{H^q},
		\quad v\in H^q(\mathbb T^d).
	\end{equation}
	Moreover, for $v,w\in H^q(\mathbb T^d)$,
	\begin{equation}\label{4.4}
		\begin{aligned}
			\Lambda^q\mathcal L_\delta v&=\mathcal L_\delta\Lambda^q v,\\
			(\mathcal L_\delta v,w)_{H^q}&=(v,\mathcal L_\delta w)_{H^q}.
		\end{aligned}
	\end{equation}
\end{lemma}
\begin{proof}
	Changing variables $z=\delta\xi$ in \eqref{2.4} and using $0\le1-\cos\theta\le\min\{\theta^2/2,2\}$ together with \eqref{K2} give \eqref{4.1}. The identity $\widehat{\mathcal L_\delta v}_k=-\lambda_\delta(k)\widehat v_k$, \eqref{Hnorm}, and the two bounds in \eqref{4.1} yield \eqref{4.2} and \eqref{4.3}. Finally, $\Lambda^q$ and $\mathcal L_\delta$ are commuting Fourier multipliers, and the symbol $-\lambda_\delta(k)$ is real. These facts give both identities in \eqref{4.4}.
\end{proof}

Next, we define the nonlinear mapping
\begin{equation}\label{4.5}
	\mathcal N(v)=f(|v|^2)v.
\end{equation}
\begin{lemma}\label{lemma4.2}
	Let $q$ be a nonnegative integer satisfying $q>d/2$ and $q\le r$, and the assumption \eqref{N1} holds. Then the mapping $\mathcal N:H^q(\mathbb T^d)\to H^q(\mathbb T^d)$ is locally Lipschitz, i.e.,  for every $R>0$, there exists $C_R>0$ such that
	\begin{equation}\label{4.6}
		\|\mathcal N(v)-\mathcal N(w)\|_{H^q}\le C_R\|v-w\|_{H^q}
	\end{equation}
	whenever $\|v\|_{H^q}+\|w\|_{H^q}\le R$.
	In addition, there exists a nondecreasing function $C_q:[0,\infty)\to(0,\infty)$ that depends on $f$, such that
	\begin{equation}\label{4.7}
		\|\mathcal N(v)\|_{H^q}\le C_q(\|v\|_{L^\infty})\|v\|_{H^q}.
	\end{equation}
	Moreover, the mapping $\mathcal N:H^q\to H^q$ is twice continuously Fréchet differentiable with respect to the real and imaginary parts. For every $R>0$, $\ell\in\{1,2\}$ and $\|v\|_{H^q}\le R$, its derivatives satisfy
	\begin{equation}
		\|D^\ell\mathcal N(v)[w_1,\ldots,w_\ell]\|_{H^q}
		\le C_R\prod_{j=1}^\ell\|w_j\|_{H^q}.
	\end{equation}
\end{lemma}
\begin{proof}
	Identify $\mathbb C$ with $\mathbb R^2$ and set $\Phi(z_1,z_2)=f(z_1^2+z_2^2)(z_1,z_2)$. Then $\mathcal N$ is the pointwise composition induced by $\Phi$. Assumption \eqref{N1} and $q\le r$ imply $\Phi\in C_{\mathrm{loc}}^{q+3}(\mathbb R^2;\mathbb R^2)$, and $\Phi(0)=0$. Hence \eqref{A1}--\eqref{A3} in Lemma \ref{lemmaA} give \eqref{4.6}, \eqref{4.7}, the stated differentiability, and the derivative bounds.
\end{proof}
\subsection{Well-posedness and conservation laws}\label{sec4.2}
For fixed $\delta>0$, the nonlocal operator is bounded on $H^q(\mathbb T^d)$. Therefore, Equation \eqref{2.5} can be treated as an ordinary differential equation in $H^q(\mathbb T^d)$. We next establish its global well-posedness for a fixed horizon.
\begin{theorem}[Global well-posedness]\label{theorem4.3}
	Let $q$ be a nonnegative integer satisfying $d/2<q\le r$, and let $u_{\delta,0}\in H^q(\mathbb T^d)$. Under the assumptions of \eqref{K1}, \eqref{K2} and \eqref{N1}, for every $\delta\in(0,\overline\delta]$, \eqref{2.5} has a unique global solution
	\begin{equation}\label{4.8}
		u_\delta\in C^1\bigl([0,\infty);H^q(\mathbb T^d)\bigr).
	\end{equation}
	Moreover, for every $T>0$ and $R>0$, there exists $C_{\delta,T,R}>0$ such that any two solutions $u_\delta$ and $v_\delta$ with
	\begin{equation}
		\|u_{\delta,0}\|_{H^q}+\|v_{\delta,0}\|_{H^q}\le R
	\end{equation}
	satisfy
	\begin{equation}
		\sup_{0\le t\le T}\|u_\delta(t)-v_\delta(t)\|_{H^q}\le C_{\delta,T,R}\|u_{\delta,0}-v_{\delta,0}\|_{H^q}.
	\end{equation}
\end{theorem}
\begin{proof}
	Equation \eqref{2.5} is equivalent to
	\begin{equation}\label{evolution equation}
		\partial_tu_\delta=\mathrm i\mathcal L_\delta u_\delta-\mathrm i\mathcal N(u_\delta).
	\end{equation}
	For fixed $\delta>0$, \eqref{4.3} in Lemma \ref{lemma4.1} and \eqref{4.6} in Lemma \ref{lemma4.2} show that the right-hand side of \eqref{evolution equation} is locally Lipschitz on $H^q$. The Banach-space ordinary differential equation theorem \cite{deimling2006ordinary} therefore gives a unique maximal solution $u_\delta\in C^1([0,T_{\max});H^q)$.
	
	To exclude $T_{\max}<\infty$, define
	\begin{equation}
		\Theta(x,t)=\int_0^t f(|u_\delta(x,\xi)|^2)\,\mathrm d\xi,
		\qquad w_\delta(x,t)=e^{\mathrm i\Theta(x,t)}u_\delta(x,t).
	\end{equation}
	Using \eqref{evolution equation},
	\begin{equation}
		\begin{aligned}
			\partial_tw_\delta
			&=\mathrm i(\partial_t\Theta)e^{\mathrm i\Theta}u_\delta+e^{\mathrm i\Theta}\partial_tu_\delta\\
			&=\mathrm i f(|u_\delta|^2)e^{\mathrm i\Theta}u_\delta+e^{\mathrm i\Theta}\left[\mathrm i\mathcal L_\delta u_\delta-\mathrm i f(|u_\delta|^2)u_\delta\right]\\
			&=\mathrm i e^{\mathrm i\Theta}\mathcal L_\delta u_\delta.
		\end{aligned}
	\end{equation}
	Since $|w_\delta|=|u_\delta|$, integration in time and \eqref{2.1} give
	\begin{equation}
		\|u_\delta(t)\|_{L^\infty}\le\|u_{\delta,0}\|_{L^\infty}+2\delta^{-2}\|\rho\|_{L^1(B_1)}\int_0^t\|u_\delta(s)\|_{L^\infty}\,\mathrm ds.
	\end{equation}
	If $T_{\max}<\infty$, Gr\"onwall's inequality gives a uniform $L^\infty$ bound on $[0,T_{\max})$. Taking the $H^q$ inner product of \eqref{evolution equation} with $u_\delta$, taking real parts, and using \eqref{4.4} in Lemma \ref{lemma4.1} and \eqref{4.7} in Lemma \ref{lemma4.2} give
	\begin{equation}
		\frac{\mathrm d}{\mathrm dt}\|u_\delta(t)\|_{H^q}^2\le2C_q(\|u_\delta(t)\|_{L^\infty})\|u_\delta(t)\|_{H^q}^2.
	\end{equation}
	A second application of Gr\"onwall's inequality gives a uniform $H^q$ bound on $[0,T_{\max})$. Equation \eqref{evolution equation}, \eqref{4.3} in Lemma \ref{lemma4.1}, and \eqref{4.7} in Lemma \ref{lemma4.2} also give $\sup_{0\le t<T_{\max}}\|\partial_tu_\delta(t)\|_{H^q}<\infty$. Hence $u_\delta(t)$ has a limit in $H^q$ as $t\uparrow T_{\max}$, and the local existence theorem extends the solution beyond $T_{\max}$, and leads to a contradiction.
	
	If the initial norms are bounded by $R$, the estimates just obtained place both solutions in a fixed $H^q$-ball depending only on $\delta$, $T$, and $R$. Subtracting the two copies of \eqref{evolution equation}, using \eqref{4.4} in Lemma \ref{lemma4.1} and \eqref{4.6} in Lemma \ref{lemma4.2}, and applying Gronwall's inequality give the stated stability estimate.
\end{proof}
\begin{theorem}[Mass and energy conservation]\label{theorem4.4}
	Under the assumptions of Theorem \ref{theorem4.3}, the solution of \eqref{2.5} satisfies
	\begin{equation}\label{4.10}
		M(u_\delta(t))=M(u_{\delta,0})
	\end{equation}
	and
	\begin{equation}\label{4.11}
		E_\delta(u_\delta(t))=E_\delta(u_{\delta,0})
	\end{equation}
	for all $t\ge0$, where $M$ and $E_\delta$ are defined in \eqref{2.7} and \eqref{2.8}.
\end{theorem}
\begin{proof}
	Taking the $L^2$ inner product of \eqref{2.5} with $u_\delta$ and taking imaginary parts gives \eqref{4.10}, because \eqref{integralbyparts} and the reality of $f$ make both terms on the right-hand side real. Taking the $L^2$ inner product of \eqref{2.5} with $\partial_tu_\delta$ and taking real parts gives \eqref{4.11} by \eqref{integralbyparts} and $F'=f$.
\end{proof}
\subsection{Consistency and nonlocal-to-local convergence}\label{sec4.3}
We next compare $\mathcal L_\delta$ with the Laplacian. The following lemma bounds the difference between their Fourier symbols uniformly in $k$ and derives the corresponding estimates for $\mathcal L_\delta-\Delta$ and $a_\delta(v,v)-\|\nabla v\|_{L^2}^2$.
\begin{lemma}\label{lemma4.5}
	Assume \eqref{K1} and \eqref{K2} hold and define
	\begin{equation}\label{m4rho}
		m_4(\rho)=\int_{B_1}|z|^4\rho(z)\,\mathrm dz,
	\end{equation}
	which is finite according to \eqref{K1}. Then there exists a constant $C_\rho>0$, depending only on $m_4(\rho)$, such that
	\begin{equation}\label{4.13}
		\bigl|\lambda_\delta(k)-|k|^2\bigr|\le C_\rho\delta^2|k|^4,\quad k\in\mathbb Z^d.
	\end{equation}
	Then for every $q\in\mathbb R$ and $v\in H^{q+4}(\mathbb T^d)$,
	\begin{equation}\label{4.14}
		\|(\mathcal L_\delta-\Delta)v\|_{H^q}\le C_\rho\delta^2\|v\|_{H^{q+4}}.
	\end{equation}
	Moreover, for every $v\in H^2(\mathbb T^d)$,
	\begin{equation}\label{4.15}
		\left|a_\delta(v,v)-\|\nabla v\|_{L^2}^2\right|\le C_\rho\delta^2\|v\|_{H^2}^2.
	\end{equation}
\end{lemma}
\begin{proof}
	Taylor's formula applied to \eqref{2.4}, after the change of variables $z=\delta\xi$, and the moment condition \eqref{K2} give \eqref{4.13} with $C_\rho=m_4(\rho)/24$. Applying \eqref{4.13} to the Sobolev norm \eqref{Hnorm} gives \eqref{4.14}. Applying \eqref{4.13} to the Fourier representation \eqref{avvpar} and using Parseval's identity gives \eqref{4.15}.
\end{proof}

Next, we assume that the regularity index $r$ in \eqref{N1} satisfies $r>\frac d2$. Let $u$ be a solution of the local NLS \eqref{2.6} satisfying
\begin{equation}\label{4.17}
	u\in C\bigl([0,T];H^{r+4}(\mathbb T^d)\bigr).
\end{equation}
Furthermore, the nonlocal and local initial values are assumed to satisfy
\begin{equation}\label{4.18}
	\|u_{\delta,0}-u_0\|_{H^r}\le C_{\mathrm{in}}\delta^2,
	\quad 0<\delta\le\overline\delta,
\end{equation}
with $C_{\mathrm{in}}$ independent of $\delta$. Then we have the following nonlocal-to-local limit theorem.
\begin{theorem}[Quantitative nonlocal-to-local limit]\label{theorem4.6}
	Assume \eqref{K1}, \eqref{K2}, \eqref{N1} and \eqref{4.17}, \eqref{4.18} hold with $r>d/2$. Then there exist $\delta_0\in(0,\overline\delta]$ and a constant $C>0$, independent of $\delta$, such that
	\begin{equation}\label{4.19}
		\sup_{0\le t\le T}\|u_\delta(t)-u(t)\|_{H^r}\le C\delta^2,\quad 0<\delta\le\delta_0.
	\end{equation}
	In particular,
	\begin{equation}\label{4.20}
		\sup_{0<\delta\le\delta_0}\|u_\delta\|_{L^\infty(0,T;H^r)}\le C.
	\end{equation}
\end{theorem}
\begin{proof}
	Set $e_\delta=u_\delta-u$. Subtracting \eqref{2.6} from \eqref{2.5} gives
	\begin{equation}
		\mathrm i\partial_te_\delta=-\mathcal L_\delta e_\delta+\mathcal N(u_\delta)-\mathcal N(u)-(\mathcal L_\delta-\Delta)u.
	\end{equation}
	Taking the $H^r$ inner product with $e_\delta$, taking real parts after multiplication by $-\mathrm i$, and using \eqref{4.4} in Lemma \ref{lemma4.1} give
	\begin{equation}
		\frac{\mathrm d}{\mathrm dt}\|e_\delta\|_{H^r}
		\le \|\mathcal N(u_\delta)-\mathcal N(u)\|_{H^r}+\|(\mathcal L_\delta-\Delta)u\|_{H^r}.
	\end{equation}
	On every interval on which $\|e_\delta(t)\|_{H^r}\le1$, assumption \eqref{4.17} bounds $u$, and $u_\delta=u+e_\delta$ then places both $u_\delta$ and $u$ in a fixed $H^r$-ball independent of $\delta$. Hence \eqref{4.6} in Lemma \ref{lemma4.2} and \eqref{4.14} in Lemma \ref{lemma4.5} yield
	\begin{equation}
		\frac{\mathrm d}{\mathrm dt}\|e_\delta(t)\|_{H^r}\le C\|e_\delta(t)\|_{H^r}+C_\rho\delta^2\|u(t)\|_{H^{r+4}}.
	\end{equation}
	By \eqref{4.17}, \eqref{4.18}, and Gronwall's inequality, there exists $C_T\ge C_{\mathrm{in}}$, independent of $\delta$, such that $\sup_{0\le \xi\le t}\|e_\delta(\xi)\|_{H^r}\le C_T\delta^2$ whenever $\sup_{0\le \xi\le t}\|e_\delta(\xi)\|_{H^r}\le1$. Choose $\delta_0\in(0,\overline\delta]$ so that $C_T\delta_0^2<1$. Then \eqref{4.18} starts the bootstrap, and continuity extends the estimate to $[0,T]$, proving \eqref{4.19}. Finally, \eqref{4.20} follows from \eqref{4.17}, \eqref{4.19}, and $\|u_\delta(t)\|_{H^r}\le\|u(t)\|_{H^r}+\|e_\delta(t)\|_{H^r}$.
\end{proof}
\begin{corollary}[Uniform time regularity]\label{corollary4.7}
	Under the assumptions of Theorem \ref{theorem4.6}, let $s$ be a nonnegative integer satisfying
	\begin{equation}
		s>\frac d2,\quad r\ge s+6.
	\end{equation}
	Then we have
	\begin{equation}\label{4.21}
		\begin{aligned}
			\sup_{0<\delta\le\delta_0}
			\Bigl(&\|u_\delta\|_{L^\infty(0,T;H^{s+6})}+\|\partial_tu_\delta\|_{L^\infty(0,T;H^{s+4})}\\
			&\quad+\|\partial_t^2u_\delta\|_{L^\infty(0,T;H^{s+2})}+\|\partial_t^3u_\delta\|_{L^\infty(0,T;H^s)}\Bigr)\le C.
		\end{aligned}
	\end{equation}
\end{corollary}
\begin{proof}
	For fixed $\delta>0$, Theorem \ref{theorem4.3}, \eqref{4.3} in Lemma \ref{lemma4.1}, and the differentiability statement in Lemma \ref{lemma4.2} imply $u_\delta\in C^3([0,T];H^r)$. Differentiating \eqref{evolution equation} twice and applying \eqref{4.2} in Lemma \ref{lemma4.1}, the bounds for $D^\ell\mathcal N$, $\ell=1,2$, in Lemma \ref{lemma4.2}, and \eqref{4.20} in Theorem \ref{theorem4.6} give uniform bounds for $u_\delta$, $\partial_tu_\delta$, $\partial_t^2u_\delta$, and $\partial_t^3u_\delta$ in $H^r$, $H^{r-2}$, $H^{r-4}$, and $H^{r-6}$, respectively. Since $r\ge s+6$, the embeddings $H^{r-2j}\hookrightarrow H^{s+6-2j}$, $0\le j\le3$, prove \eqref{4.21}.
\end{proof}
\begin{corollary}[Convergence of the energy]\label{corollary 4.8}
	Under the assumptions of Theorem \ref{theorem4.6}, we have
	\begin{equation}\label{4.22}
		\sup_{0\le t\le T}\bigl|E_\delta(u_\delta(t))-E_0(u(t))\bigr|\le C\delta^2,\quad 0<\delta\le\delta_0.
	\end{equation}
\end{corollary}
\begin{proof}
	For $v,w\in H^1(\mathbb T^d)$, \eqref{4.1} in Lemma \ref{lemma4.1} and \eqref{avvpar} give
	\begin{equation}\label{shuangxiannxing}
		|a_\delta(v,w)|\le\|v\|_{H^1}\|w\|_{H^1}.
	\end{equation}
	Thus \eqref{4.19} in Theorem \ref{theorem4.6}, \eqref{4.15} in Lemma \ref{lemma4.5}, and \eqref{shuangxiannxing} imply
	\begin{equation}
		|a_\delta(u_\delta,u_\delta)-\|\nabla u\|_{L^2}^2|\le C\delta^2.
	\end{equation}
	Moreover, assumption \eqref{4.17}, the bound \eqref{4.20} in Theorem \ref{theorem4.6}, the embedding $H^r\hookrightarrow L^\infty$, and the mean-value theorem give
	\begin{equation}
		\frac1{(2\pi)^d}\left|\int_{\mathbb T^d}\bigl(F(|u_\delta|^2)-F(|u|^2)\bigr)\,\mathrm dx\right|
		\le C\|u_\delta-u\|_{L^2}\le C\delta^2.
	\end{equation}
	Combining these two estimates with \eqref{2.8} and \eqref{2.9} proves \eqref{4.22}.
\end{proof}

\subsection{Dispersion properties of the nonlocal NLS model}\label{sec4.4}
The Fourier multiplier $\lambda_\delta(k)$ determines the linear frequency contribution associated with the mode $e^{\mathrm i k\cdot x}$. To study both the dispersion relation and the corresponding group velocity, we extend $\lambda_\delta$ from the discrete wave vectors $k\in\mathbb Z^d$ to $\xi\in\mathbb R^d$ by
\begin{equation}\label{4.23}
	\lambda_\delta(\xi)=\delta^{-2}\int_{B_1}\rho(z)\bigl(1-\cos(\delta\xi\cdot z)\bigr)\,\mathrm dz.
\end{equation}

At $\xi=k\in\mathbb Z^d$, \eqref{4.23} agrees with the Fourier multiplier \eqref{2.4}. The cosine term in \eqref{4.23} depends on $\xi$ only through the scaled wave vector $\delta\xi$, while the prefactor $\delta^{-2}$ determines the overall scale of the multiplier. If $\rho$ is radial, the integral depends on $\xi$ only through $\delta|\xi|$. The following proposition describes the local approximation for small $\delta|\xi|$ and the high-frequency behavior for fixed $\delta$.
\begin{proposition}[Low- and high-frequency behavior]\label{proposition4.9}
	Assume \eqref{K1} and \eqref{K2}. Recall the fourth moment $m_4(\rho)$ defined in \eqref{m4rho} and set 
	\begin{equation}
		m_0(\rho)=\int_{B_1}\rho(z)\,\mathrm dz.
	\end{equation}
	Then for every $0<\delta\le\overline\delta$, the function $\lambda_\delta$ is even, nonnegative, and continuously differentiable, with
	\begin{equation}\label{4.24}
		\nabla_\xi\lambda_\delta(\xi)=\delta^{-1}\int_{B_1}\rho(z)z\sin(\delta\xi\cdot z)\,\mathrm dz.
	\end{equation}
	Moreover,
	\begin{equation}\label{4.25}
		\begin{aligned}
			\bigl|\lambda_\delta(\xi)-|\xi|^2\bigr|
			&\le
			\frac{m_4(\rho)}{24}|\xi|^2(\delta|\xi|)^2,\\
			\bigl|\nabla_\xi\lambda_\delta(\xi)-2\xi\bigr|
			&\le
			\frac{m_4(\rho)}{6}|\xi|(\delta|\xi|)^2,
		\end{aligned}
		\qquad \xi\in\mathbb R^d.
	\end{equation}
	And for each fixed $\delta>0$, we have
	\begin{equation}\label{4.26}
		\lambda_\delta(\xi)\longrightarrow\delta^{-2}m_0(\rho),\quad
		\nabla_\xi\lambda_\delta(\xi)\longrightarrow0,
		\quad |\xi|\to\infty.
	\end{equation}
\end{proposition}
\begin{proof}
	The evenness and nonnegativity follow from \eqref{4.23}, and differentiation under the integral sign gives \eqref{4.24}. Taylor's formula, \eqref{K2}, and the definition \eqref{m4rho} give both estimates in \eqref{4.25}. Finally, rewriting \eqref{4.23} as $\lambda_\delta(\xi)=\delta^{-2}m_0(\rho)-\delta^{-2}\int_{B_1}\rho(z)\cos(\delta\xi\cdot z)\,\mathrm dz$ and applying the Riemann--Lebesgue lemma to the zero extensions of $\rho$ and $z_j\rho$, $1\le j\le d$, give \eqref{4.26}.
\end{proof}

The following corollary gives the corresponding dispersion relation and group velocity for the nonlocal NLS equation.

\begin{corollary}[Dispersion relation and group velocity]\label{corollary4.10}
	Let $A\in\mathbb C$ and $k\in\mathbb Z^d$. For the initial value $u_{\delta,0}(x)=A e^{\mathrm i k\cdot x}$, the solution of \eqref{2.5} is
	\begin{equation}
		u_\delta(x,t)=A\exp\!\left(\mathrm i\bigl(k\cdot x-\omega_\delta(k;A)t\bigr)\right),
	\end{equation}
	where
	\begin{equation}\label{4.27}
		\omega_\delta(k;A)=\lambda_\delta(k)+f(|A|^2).
	\end{equation}
	Moreover, we extend the nonlocal and local dispersion relations to $\xi\in\mathbb R^d$ by
	\begin{equation}
		\omega_\delta(\xi;A)=\lambda_\delta(\xi)+f(|A|^2),
		\quad
		\omega_0(\xi;A)=|\xi|^2+f(|A|^2),
	\end{equation}
	and the corresponding group velocities are defined by
	\begin{equation}
		v_{g,\delta}(\xi;A)=\nabla_\xi\omega_\delta(\xi;A)=\nabla_\xi\lambda_\delta(\xi),
		\quad
		v_{g,0}(\xi;A)=2\xi.
	\end{equation}
	Then, for every $\xi\in\mathbb R^d$,
	\begin{equation}\label{4.28}
		\begin{aligned}
			\bigl|\omega_\delta(\xi;A)-\omega_0(\xi;A)\bigr|&\le\frac{m_4(\rho)}{24}\delta^2|\xi|^4,\\\bigl|v_{g,\delta}(\xi;A)-v_{g,0}(\xi;A)\bigr|
			&\le\frac{m_4(\rho)}{6}\delta^2|\xi|^3.
		\end{aligned}
	\end{equation}
	And for fixed $\delta>0$,
	\begin{equation}\label{4.29}
		\omega_\delta(\xi;A)\longrightarrow\delta^{-2}m_0(\rho)+f(|A|^2),
		\quad
		v_{g,\delta}(\xi;A)\longrightarrow0,
		\quad |\xi|\to\infty.
	\end{equation}
\end{corollary}
\begin{proof}
	Since $|u_\delta|=|A|$ and \eqref{2.3} gives $-\mathcal L_\delta e_k=\lambda_\delta(k)e_k$, substitution into \eqref{2.5} gives \eqref{4.27}; uniqueness follows from Theorem \ref{theorem4.3}. The identities defining $\omega_\delta$, $\omega_0$, $v_{g,\delta}$, and $v_{g,0}$ show that \eqref{4.28} follows from \eqref{4.25} in Proposition \ref{proposition4.9}, while \eqref{4.29} follows from \eqref{4.26} in Proposition \ref{proposition4.9}.
\end{proof}
\begin{remark}
	For every fixed $R>0$, \eqref{4.28} implies uniform convergence of the dispersion relation and group velocity on $\{|\xi|\le R\}$ as $\delta\to0$. This convergence is not uniform on $\mathbb R^d$. Indeed, for fixed $\delta$, \eqref{4.29} shows that the nonlocal frequency approaches a finite limit and the nonlocal group velocity tends to zero as $|\xi|\to\infty$, whereas their local counterparts grow without bound. These two regimes will be examined numerically for different horizons and kernel profiles.
\end{remark}
\section{Temporal discretization analysis}\label{sec5}
In this section, we analyze the time-discrete scheme \eqref{3.6}. Throughout the section, $s$ and $r$ satisfy the assumptions of Corollary \ref{corollary4.7}, and $0<\delta\le\delta_0$. Unless stated otherwise, all constants are independent of $\delta$ and $\tau$. We first establish the one-step solvability of the time-discrete scheme, then we prove its conservation laws and derive the uniform temporal error and regularity estimates required in the fully discrete analysis.
\subsection{One-step solvability}
The following lemma establishes the local Lipschitz continuity of $\mathscr G$ defined in \eqref{3.5} and its second-order consistency with the continuous nonlinearity $\mathcal N(v)=f(|v|^2)v$ defined in \eqref{4.5}.
\begin{lemma}\label{lemma5.1}
	For every $R>0$, there exists $C_R>0$ such that
	\begin{equation}\label{5.1}
		\begin{aligned}
			&\|\mathscr G(v_1,w_1)-\mathscr G(v_2,w_2)\|_{H^s}\\
			&\qquad\le C_R\bigl(\|v_1-v_2\|_{H^s}+\|w_1-w_2\|_{H^s}\bigr)
		\end{aligned}
	\end{equation}
	whenever
	\begin{equation}
		\|v_1\|_{H^s}+\|w_1\|_{H^s}+\|v_2\|_{H^s}+\|w_2\|_{H^s}\le R.
	\end{equation}
	Moreover,
	\begin{equation}\label{5.2}
		\left\|\mathscr G(v,w)-\mathcal N\!\left(\frac{v+w}{2}\right)\right\|_{H^s}\le C_R\|v-w\|_{H^s}^2
	\end{equation}
	whenever $\|v\|_{H^s}+\|w\|_{H^s}\le R$.
\end{lemma}
\begin{proof}
	For $z,\zeta\in\mathbb C$, set
	\begin{equation}
		\Psi(z,\zeta)=G(|z|^2,|\zeta|^2)\frac{z+\zeta}{2}
		=\left[\int_0^1 f\bigl((1-\theta)|\zeta|^2+\theta|z|^2\bigr)\,\mathrm d\theta\right]\frac{z+\zeta}{2},
	\end{equation}
	so that $\mathscr G(v,w)(x)=\Psi(v(x),w(x))$. Assumption \eqref{N1} and Lemma \ref{lemmaA} show that the mapping $(v,w)\mapsto\mathscr G(v,w)$ is twice continuously Fr\'{e}chet differentiable on bounded subsets of $H^s\times H^s$, while \eqref{A2} gives bounded first and second derivatives there. The mean-value formula therefore gives \eqref{5.1}.
	
	For \eqref{5.2}, set $v_{\mathrm{av}}=(v+w)/2$, $v_{\mathrm{dif}}=(v-w)/2$, and $\phi(\theta)=\mathscr G(v_{\mathrm{av}}+\theta v_{\mathrm{dif}},v_{\mathrm{av}}-\theta v_{\mathrm{dif}})$. Since $G(a,b)=G(b,a)$, $\phi$ is even and $\phi'(0)=0$. Estimate \eqref{A2} in Lemma \ref{lemmaA} gives $\|\phi''(\theta)\|_{H^s}\le C_R\|v-w\|_{H^s}^2$. Taylor's formula gives \eqref{5.2}, because $\phi(1)=\mathscr G(v,w)$ and $\phi(0)=\mathcal N(v_{\mathrm{av}})$.
\end{proof}

The following lemma provides the linear estimates needed for one-step solvability.

\begin{lemma}\label{lemmaB}
	Let $m\in\mathbb R$, $\delta>0$, and $\tau>0$, and define
	\begin{equation}
		\mathcal S_{\delta,\tau}
		=\left(\frac{\mathrm i}{\tau}I+\frac12\mathcal L_\delta\right)^{-1},
		\quad
		\mathcal C_{\delta,\tau}
		=\mathcal S_{\delta,\tau}\left(\frac{\mathrm i}{\tau}I-\frac12\mathcal L_\delta\right).
	\end{equation}
	Then for every $v\in H^m(\mathbb T^d)$, we have
	\begin{equation}\label{5.4}
		\|\mathcal S_{\delta,\tau}v\|_{H^m}\le\tau\|v\|_{H^m},
	\end{equation}
	and
	\begin{equation}\label{5.5}
		\|\mathcal C_{\delta,\tau}v\|_{H^m}=\|v\|_{H^m}.
	\end{equation}
\end{lemma}
\begin{proof}
	The Fourier symbols of $\mathcal S_{\delta,\tau}$ and $\mathcal C_{\delta,\tau}$ are
	\begin{equation}
		\left(\frac{\mathrm i}{\tau}-\frac12\lambda_\delta(k)\right)^{-1},
		\quad
		\frac{\mathrm i/\tau+\lambda_\delta(k)/2}{\mathrm i/\tau-\lambda_\delta(k)/2}.
	\end{equation}
	Since $\lambda_\delta(k)$ is real, their moduli are bounded by $\tau$ and equal to $1$, respectively. Using \eqref{Hnorm} with $q=m$ gives \eqref{5.4} and \eqref{5.5}.
\end{proof}

Let $n\ge0$ and suppose that $U_\delta^n$ is given. We next show that the time-discrete scheme \eqref{3.6} uniquely determines $U_\delta^{n+1}$ for sufficiently small $\tau$.

\begin{proposition}[One-step solvability] \label{proposition5.2}
	For every $R>0$, there exists $\tau_R>0$, independent of $0<\delta\le\delta_0$, such that for $\|U_\delta^n\|_{H^s}\le R$ and $0<\tau\le\tau_R$, the time-discrete scheme \eqref{3.6} has a unique solution
	$U_\delta^{n+1}$ in the closed ball
	\begin{equation}
		B_{R+1}=\{W\in H^s(\mathbb T^d):\|W\|_{H^s}\le R+1\}.
	\end{equation}
\end{proposition}
\begin{proof}
	Rearranging \eqref{3.6} gives the fixed-point equation
	\begin{equation}\label{5.6}
		W=\mathcal T_n(W):=\mathcal C_{\delta,\tau}U_\delta^n+\mathcal S_{\delta,\tau}\mathscr G(W,U_\delta^n).
	\end{equation}
	For $W,Z\in B_{R+1}$, estimate \eqref{5.1} in Lemma \ref{lemma5.1}, $\mathscr G(0,0)=0$, and \eqref{5.4}--\eqref{5.5} in Lemma \ref{lemmaB} give
	\begin{equation}
		\|\mathcal T_n(W)\|_{H^s}\le R+\tau K_R,
		\qquad
		\|\mathcal T_n(W)-\mathcal T_n(Z)\|_{H^s}
		\le\tau C_R\|W-Z\|_{H^s},
	\end{equation}
	where $K_R$ and $C_R$ are independent of $\delta$. Choose $\tau_R>0$ so that $\tau_RK_R\le1$ and $\tau_RC_R<1$. Then $\mathcal T_n$ is a contraction from $B_{R+1}$ into itself, and its unique fixed point is the required solution.
\end{proof}

\subsection{Conservation, consistency, and convergence}
The first result shows that the time discretization preserves the mass and the nonlocal energy defined in section \ref{sec2}.
\begin{theorem}[Conservation laws of the time-discrete scheme]\label{theorem5.3}
	Let $\{U_\delta^n\}_{n=0}^{N_T}\subset H^s(\mathbb T^d)$ satisfy \eqref{3.6}. Then, for $0\le n\le N_T-1$,
	\begin{equation}\label{5.7}
		M(U_\delta^{n+1})=M(U_\delta^n),
	\end{equation}
	and
	\begin{equation}\label{5.8}
		E_\delta(U_\delta^{n+1})=E_\delta(U_\delta^n).
	\end{equation}
\end{theorem}
\begin{proof}
	Taking the $L^2$ inner product of \eqref{3.6} with $U_\delta^{n+1/2}$, taking imaginary parts, and using \eqref{integralbyparts} and \eqref{3.2} give \eqref{5.7}. Taking the $L^2$ inner product of \eqref{3.6} with $D_\tau U_\delta^n$, taking real parts, and again using \eqref{integralbyparts} and \eqref{3.2} give \eqref{5.8}.
\end{proof}

The following two lemmas give the midpoint energy estimate and the discrete Gr{\"o}nwall argument used in Sections \ref{sec5} and \ref{sec6}.

\begin{lemma}[Midpoint energy estimate]\label{lemmaC}
	Let $\delta>0$ and $\{e^n\}\subset H^s(\mathbb T^d)$, and let $Q^{n+1/2}\in H^s(\mathbb T^d)$. Set $D_\tau e^n=(e^{n+1}-e^n)/\tau$ and $e^{n+1/2}=(e^{n+1}+e^n)/2$, and further suppose
	\begin{equation}\label{lemmacerror}
		\mathrm iD_\tau e^n=-\mathcal L_\delta e^{n+1/2}+Q^{n+1/2}.
	\end{equation}
	If it holds that
	\begin{equation}\label{assumedboundlemmac}
		\|Q^{n+1/2}\|_{H^s}\le L\bigl(\|e^{n+1}\|_{H^s}+\|e^n\|_{H^s}\bigr)+\varepsilon_n,
	\end{equation}
	then we have
	\begin{equation}
		\|e^{n+1}\|_{H^s}^2-\|e^n\|_{H^s}^2\le C_L\tau\bigl(\|e^{n+1}\|_{H^s}^2+\|e^n\|_{H^s}^2\bigr)+C\tau\varepsilon_n^2.
	\end{equation}
\end{lemma}
\begin{proof}
	Taking the $H^s$ inner product of \eqref{lemmacerror} with $e^{n+1/2}$, taking imaginary parts, and using \eqref{4.4} in Lemma \ref{lemma4.1} give
	\begin{equation}
		\frac{\|e^{n+1}\|_{H^s}^2-\|e^n\|_{H^s}^2}{2\tau}=\operatorname{Im}(Q^{n+1/2},e^{n+1/2})_{H^s}.
	\end{equation}
	Using \eqref{assumedboundlemmac}, $\|e^{n+1/2}\|_{H^s}\le(\|e^{n+1}\|_{H^s}+\|e^n\|_{H^s})/2$, and Young's inequality gives the stated estimate.
\end{proof}

\begin{lemma}[Discrete Gr{\"o}nwall estimate]\label{lemmaD}
	Let $T,C,C_0>0$, $\tau>0$, $\varepsilon\ge0$, and $m\in\mathbb N$ satisfy $m\tau\le T$. Let $a_0,\ldots,a_m$ be nonnegative numbers such that
	\begin{equation}\label{discreteGA}
		a_{n+1}-a_n\le C\tau(a_{n+1}+a_n)+C\tau\varepsilon^2,
		\quad
		0\le n\le m-1,
	\end{equation}
	and $a_0\le C_0\varepsilon^2$. Then there exist $\tau_*>0$ and $C_T>0$, depending only on $C$, $C_0$, and $T$, such that if $0<\tau\le\tau_*$,
	\begin{equation}
		\max_{0\le j\le m}a_j^{1/2}\le C_T\varepsilon.
	\end{equation}
\end{lemma}
\begin{proof}
	Choose $\tau_*>0$ so that $C\tau_*\le1/2$. For $0<\tau\le\tau_*$, \eqref{discreteGA} implies $a_{n+1}\le(1+C_1\tau)a_n+C_1\tau\varepsilon^2$, where $C_1$ depends only on $C$. The discrete Gr{\"o}nwall inequality, $a_0\le C_0\varepsilon^2$, and $m\tau\le T$ give $\max_{0\le j\le m}a_j\le C_T^2\varepsilon^2$.
\end{proof}

We next establish the consistency of the time-discrete scheme. Let
\begin{equation}
	u_\delta^n=u_\delta(t_n),
\end{equation}
and define the local residual by
\begin{equation}\label{5.9}
	\mathcal R_\tau^{n+\frac12}=\mathrm iD_\tau u_\delta^n+\mathcal L_\delta\frac{u_\delta^{n+1}+u_\delta^n}{2}-\mathscr G(u_\delta^{n+1},u_\delta^n).
\end{equation}
The following lemma establishes the second-order consistency of the time discretization, uniformly with respect to $\delta$.
\begin{lemma}[Uniform temporal consistency]\label{lemma5.4}
	Under the assumptions of Corollary \ref{corollary4.7}, we have
	\begin{equation}\label{5.10}
		\|\mathcal R_\tau^{n+\frac12}\|_{H^s}\le C\tau^2,\quad 0\le n\le N_T-1.
	\end{equation}
\end{lemma}
\begin{proof}
	Evaluating \eqref{2.5} at $t_{n+\frac12}$ and subtracting it from \eqref{5.9} give
	\begin{equation}\label{5.11}
		\begin{aligned}
			\mathcal R_\tau^{n+1/2}={}&\mathrm i\left(D_\tau u_\delta^n-\partial_tu_\delta(t_{n+1/2})\right)\\
			&+\mathcal L_\delta\left(\frac{u_\delta^{n+1}+u_\delta^n}{2}-u_\delta(t_{n+1/2})\right)\\
			&-\left[\mathscr G(u_\delta^{n+1},u_\delta^n)-\mathcal N(u_\delta(t_{n+1/2}))\right].
		\end{aligned}
	\end{equation}
	Taylor's formula and \eqref{4.21} in Corollary \ref{corollary4.7} give
	\begin{equation}
		\left\|D_\tau u_\delta^n-\partial_tu_\delta(t_{n+1/2})\right\|_{H^s}
		+\left\|\frac{u_\delta^{n+1}+u_\delta^n}{2}-u_\delta(t_{n+1/2})\right\|_{H^{s+2}}
		\le C\tau^2.
	\end{equation}
	Hence \eqref{4.2} in Lemma \ref{lemma4.1} bounds the term containing $\mathcal L_\delta$ in \eqref{5.11} by $C\tau^2$ in $H^s$. If $u_{\delta,\mathrm{av}}^{n+1/2}=(u_\delta^{n+1}+u_\delta^n)/2$, then \eqref{5.2} in Lemma \ref{lemma5.1}, \eqref{4.6} in Lemma \ref{lemma4.2}, and \eqref{4.21} in Corollary \ref{corollary4.7} give
	\begin{equation}
		\begin{aligned}
			&\|\mathscr G(u_\delta^{n+1},u_\delta^n)-\mathcal N(u_\delta(t_{n+1/2}))\|_{H^s}\\
			&\qquad\le C\|u_\delta^{n+1}-u_\delta^n\|_{H^s}^2
			+C\|u_{\delta,\mathrm{av}}^{n+1/2}-u_\delta(t_{n+1/2})\|_{H^s}
			\le C\tau^2.
		\end{aligned}
	\end{equation}
	The three bounds for the terms in \eqref{5.11} prove \eqref{5.10}.
\end{proof}

We now establish the unique solvability of the time-discrete scheme at all time levels, its second-order convergence, and a uniform $H^r$-bound for the discrete solution.
\begin{theorem}[Uniform time-discrete approximation]\label{theorem5.6new}
	Let $U_\delta^0=u_{\delta,0}$. There exist $R>0$, $\tau_0>0$, and $C>0$, independent of $\delta$ and $\tau$, such that for every $0<\delta\le\delta_0$ and $0<\tau\le\tau_0$, the time-discrete scheme \eqref{3.6} has a solution sequence $\{U_\delta^n\}_{n=0}^{N_T}$ satisfying
	\begin{equation}\label{5.20}
		\max_{0\le n\le N_T}\|U_\delta^n\|_{H^s}\le R,
	\end{equation}
	\begin{equation}\label{5.21}
		\max_{0\le n\le N_T}\|u_\delta(t_n)-U_\delta^n\|_{H^s}
		\le C\tau^2,
	\end{equation}
	and
	\begin{equation}\label{5.26}
		\max_{0\le n\le N_T}\|U_\delta^n\|_{H^r}\le C.
	\end{equation}
	Moreover, the solution is unique among all solution sequences with the prescribed initial value that satisfy \eqref{5.20}.
\end{theorem}
\begin{proof}
	By \eqref{4.21} in Corollary \ref{corollary4.7},
	\begin{equation}
		K:=\sup_{0<\delta\le\delta_0}\sup_{0\le t\le T}\|u_\delta(t)\|_{H^s}<\infty.
	\end{equation}
	Set $R=K+1$ and $e^j=u_\delta(t_j)-U_\delta^j$. We prove existence and \eqref{5.21} by induction. Since $U_\delta^0=u_{\delta,0}$, $e^0=0$. Suppose that $U_\delta^0,\ldots,U_\delta^n$ have been constructed and
	\begin{equation}
		\max_{0\le j\le n}\|e^j\|_{H^s}\le1.
	\end{equation}
	Then $\|U_\delta^n\|_{H^s}\le R$, so Proposition \ref{proposition5.2} constructs a unique $U_\delta^{n+1}\in B_{R+1}$ for $\tau\le\tau_R$.
	
	For every completed step $0\le j\le n$, subtracting \eqref{3.6} from the residual equation \eqref{5.9} gives
	\begin{equation}\label{errorequationin5.6}
		\mathrm iD_\tau e^j=-\mathcal L_\delta e^{j+1/2}+\mathcal Q^{j+1/2},
	\end{equation}
	where
	\begin{equation}
		\mathcal Q^{j+1/2}
		=\mathscr G(u_\delta^{j+1},u_\delta^j)
		-\mathscr G(U_\delta^{j+1},U_\delta^j)
		+\mathcal R_\tau^{j+1/2}.
	\end{equation}
	The definition of $K$, the induction hypothesis, and Proposition \ref{proposition5.2} show that the total $H^s$ norm of the four arguments of $\mathscr G$ is bounded by $4(R+1)$. Therefore, applying \eqref{5.1} in Lemma \ref{lemma5.1} with radius $4(R+1)$ and using \eqref{5.10} in Lemma \ref{lemma5.4} give
	\begin{equation}
		\|\mathcal Q^{j+1/2}\|_{H^s}
		\le C_R\bigl(\|e^{j+1}\|_{H^s}+\|e^j\|_{H^s}\bigr)+C\tau^2.
	\end{equation}
	Applying Lemma \ref{lemmaC} to \eqref{errorequationin5.6} yields
	\begin{equation}
		\|e^{j+1}\|_{H^s}^2-\|e^j\|_{H^s}^2
		\le C\tau\bigl(\|e^{j+1}\|_{H^s}^2+\|e^j\|_{H^s}^2\bigr)+C\tau^5.
	\end{equation}
	Apply Lemma \ref{lemmaD} with $a_j=\|e^j\|_{H^s}^2$, $\varepsilon=\tau^2$, and $a_0=0$. Let $\tau_*>0$ and $C_T>0$ be the constants in Lemma \ref{lemmaD}, and choose $\tau_0\le\min\{\tau_R,\tau_*\}$ so that $C_T\tau_0^2\le1$. Lemma \ref{lemmaD}, applied with $m=n+1$, gives
	\begin{equation}
		\max_{0\le j\le n+1}\|e^j\|_{H^s}\le C_T\tau^2\le1.
	\end{equation}
	This verifies the induction hypothesis at $t_{n+1}$. Thus the solution is defined at every time level, and taking $n=N_T-1$ proves \eqref{5.21}. The definition of $K$ and \eqref{5.21} give \eqref{5.20}. If another solution sequence satisfies \eqref{5.20} and agrees with $\{U_\delta^j\}$ through $t_n$, Proposition \ref{proposition5.2} gives equality at $t_{n+1}$. Hence the solution is unique among the sequences satisfying \eqref{5.20}.
	
	It remains to prove \eqref{5.26}. The pointwise map $\Psi$ defined in the proof of Lemma \ref{lemma5.1} satisfies $\Psi(0,0)=0$. Since $s>d/2$, the embedding $H^s\hookrightarrow L^\infty$ and \eqref{A3} in Lemma \ref{lemmaA} give
	\begin{equation}\label{5.27}
		\|\mathscr G(v,w)\|_{H^r}
		\le C_R\bigl(\|v\|_{H^r}+\|w\|_{H^r}\bigr),
		\qquad
		\|v\|_{H^s}+\|w\|_{H^s}\le2R+1.
	\end{equation}
	Assume $U_\delta^n\in H^r$ and consider the fixed-point iteration associated with \eqref{5.6},
	\begin{equation}
		W^{(0)}=U_\delta^n,
		\qquad
		W^{(m+1)}=\mathcal T_n(W^{(m)}).
	\end{equation}
	The proof of Proposition \ref{proposition5.2} gives $W^{(m)}\in B_{R+1}$ and $W^{(m)}\to U_\delta^{n+1}$ in $H^s$. Estimates \eqref{5.4}--\eqref{5.5} in Lemma \ref{lemmaB}, the fixed-point equation \eqref{5.6}, and \eqref{5.27} show inductively that $W^{(m)}\in H^r$ and give
	\begin{equation}
		\|W^{(m+1)}\|_{H^r}
		\le(1+C_R\tau)\|U_\delta^n\|_{H^r}
		+C_R\tau\|W^{(m)}\|_{H^r}.
	\end{equation}
	After decreasing $\tau_0$ so that $C_R\tau_0\le1/2$, iteration in $m$ yields
	\begin{equation}
		\sup_{m\ge0}\|W^{(m)}\|_{H^r}
		\le(1+C_1\tau)\|U_\delta^n\|_{H^r}.
	\end{equation}
	The convergence in $H^s$, Fatou's lemma, and the norm identity \eqref{Hnorm} therefore imply
	\begin{equation}
		\|U_\delta^{n+1}\|_{H^r}\le(1+C_1\tau)\|U_\delta^n\|_{H^r}.
	\end{equation}
	Iterating this estimate in $n$ and using \eqref{4.20} in Theorem \ref{theorem4.6} at $t=0$ prove \eqref{5.26}.
\end{proof}

\section{Convergence and asymptotic compatibility of the fully discrete Fourier methods}\label{sec6}
In this section, we analyze the two spatial discretizations introduced in section \ref{sec3}. The Fourier--Galerkin solution is used as an intermediate approximation between the time-discrete solution and the Fourier collocation solution. Throughout this section, we retain the regularity assumptions of Corollary \ref{corollary4.7} that
\begin{equation}
	s>\frac d2,
	\quad
	r\ge s+6,
	\quad
	s\in\mathbb N_0,
\end{equation}
and assume that $0<\delta\le\delta_0$. The constants below are independent of $\delta$, $\tau$, and $N$ unless a dependence is stated explicitly.
\subsection{Preliminary estimates}
For a grid function $V=(V_j)_{j\in\mathbb J_M}$, define its trigonometric reconstruction by
\begin{equation}\label{6.1}
	\mathcal I_NV=\sum_{k\in K_N}\widehat V_k e_k\in X_N.
\end{equation}
Moreover, for a continuous periodic function $v$, define its trigonometric interpolant by
\begin{equation}\label{6.2}
	I_Nv=\mathcal I_N\bigl((v(x_j))_{j\in\mathbb J_M}\bigr).
\end{equation}
For $q\ge0$, define the discrete Sobolev inner product and norm by
\begin{equation}\label{6.3}
	\langle V,W\rangle_{q,h}:=\sum_{k\in K_N}(1+|k|^2)^q\widehat V_k\overline{\widehat W_k},
	\quad
	\|V\|_{q,h}^2:=\langle V,V\rangle_{q,h}=\|\mathcal I_NV\|_{H^q}^2.
\end{equation}
The definitions in \eqref{3.17} -- \eqref{3.20} and \eqref{3.26}, together with the reality of $\lambda_\delta(k)$, yield
\begin{equation}\label{6.4}
	\begin{aligned}
		\mathcal I_N\mathcal L_{\delta,N}^{c}V&=
		\mathcal L_\delta\mathcal I_NV,\\
		\mathcal I_N\mathscr G_h(V,W)&=
		I_N\mathscr G(\mathcal I_NV,\mathcal I_NW),\\
		\langle\mathcal L_{\delta,N}^{c}V,W\rangle_{q,h}&=\langle V,\mathcal L_{\delta,N}^{c}W\rangle_{q,h}.
	\end{aligned}
\end{equation}

\begin{lemma}[Projection and interpolation estimates]\label{lemma6.1}
	Let $0\le q\le\mu$ and $\mu>d/2$. Then, for every $v\in H^\mu(\mathbb T^d)$,
	\begin{equation}\label{6.5}
		\|v-P_Nv\|_{H^q}+\|v-I_Nv\|_{H^q}+\|(P_N-I_N)v\|_{H^q} \le CN^{q-\mu}\|v\|_{H^\mu}.
	\end{equation}
	Moreover, for every $v\in H^\mu(\mathbb T^d)$,
	\begin{equation}\label{6.6}
		P_N\mathcal L_\delta v=\mathcal L_\delta P_Nv.
	\end{equation}
\end{lemma}

\begin{proof}
	The projection estimate in \eqref{6.5} follows directly from the Fourier definition of $P_N$:
	\begin{equation}
		\|v-P_Nv\|_{H^q}\le CN^{q-\mu}\|v\|_{H^\mu}.
	\end{equation}
	For the interpolation error, the assumption $\mu>d/2$ implies absolute convergence of the Fourier series of $v$. Discrete Fourier orthogonality therefore gives
	\begin{equation}
		\widehat{I_Nv}_k=\widehat v_k+\sum_{\ell\in\mathbb Z^d\setminus\{0\}}\sigma_\ell\widehat v_{k+M\ell},\quad k\in K_N,
		\qquad |\sigma_\ell|=1.
	\end{equation}
	Since $M=2N+1$, for $k\in K_N$ and $\ell\ne0$ we have $|k+M\ell|\ge cM|\ell|$. Hence, by Cauchy--Schwarz and $\mu>d/2$,
	\begin{equation}
		\sum_{\ell\ne0}(1+|k+M\ell|^2)^{-\mu}\le CN^{-2\mu}.
	\end{equation}
	Using the unique representation of every Fourier index as $k+M\ell$ with $k\in K_N$, we obtain
	\begin{equation}
		\|P_Nv-I_Nv\|_{H^q}
		\le CN^{q-\mu}\|v\|_{H^\mu}.
	\end{equation}
	Combining this estimate with the projection estimate proves \eqref{6.5}. Finally, $P_N$ and $\mathcal L_\delta$ are Fourier multipliers, so their symbols commute. This proves \eqref{6.6}.
\end{proof}

The next lemma gives the discrete nonlinear estimates used in the collocation analysis.
\begin{lemma}\label{lemma6.2}
	Let $V_\nu,W_\nu$, $\nu=1,2$, be grid functions. For every $R>0$, there exists $C_R>0$, independent of $N$, such that, whenever $\sum_{\nu=1}^2\left(\|V_\nu\|_{s,h}+\|W_\nu\|_{s,h}\right)\le R$,
	\begin{equation}\label{6.10}
		\|\mathscr G_h(V_1,W_1)-\mathscr G_h(V_2,W_2)\|_{s,h}\le C_R\left(\|V_1-V_2\|_{s,h}+\|W_1-W_2\|_{s,h}\right).
	\end{equation}
	Moreover, if $v_N,w_N\in X_N$ satisfy $\|v_N\|_{H^s}+\|w_N\|_{H^s}\le R$, then we have
	\begin{equation}\label{6.11}
		\left\|(P_N-I_N)\mathscr G(v_N,w_N)\right\|_{H^s}\le C_RN^{s-r}\left(\|v_N\|_{H^r}+\|w_N\|_{H^r}\right).
	\end{equation}
\end{lemma}
\begin{proof}
	Set $v_{\nu,N}=\mathcal I_NV_\nu$ and $w_{\nu,N}=\mathcal I_NW_\nu$ for $\nu=1,2$. By \eqref{6.3},
	\begin{equation}
		\|v_{\nu,N}\|_{H^s}=\|V_\nu\|_{s,h},
		\qquad
		\|w_{\nu,N}\|_{H^s}=\|W_\nu\|_{s,h}.
	\end{equation}
	Taking $q=\mu=s$ in \eqref{6.5} in Lemma \ref{lemma6.1} shows that $I_N$ is uniformly bounded on $H^s$. Hence, using the reconstruction identity \eqref{6.4} and the Lipschitz estimate \eqref{5.1} in Lemma \ref{lemma5.1}, we obtain
	\begin{equation}
		\begin{aligned}
			&\|\mathscr G_h(V_1,W_1)-\mathscr G_h(V_2,W_2)\|_{s,h}\\
			&\quad=\left\|I_N\!\left[\mathscr G(v_{1,N},w_{1,N})-\mathscr G(v_{2,N},w_{2,N})\right]\right\|_{H^s}\\
			&\quad\le C_R\left(\|V_1-V_2\|_{s,h}+\|W_1-W_2\|_{s,h}\right),
		\end{aligned}
	\end{equation}
	which proves \eqref{6.10}.
	Applying \eqref{6.5} in Lemma \ref{lemma6.1} with $q=s$ and $\mu=r$, followed by \eqref{5.27}, gives
	\begin{equation}
		\begin{aligned}
			\|(P_N-I_N)\mathscr G(v_N,w_N)\|_{H^s}
			&\le CN^{s-r}\|\mathscr G(v_N,w_N)\|_{H^r}\\
			&\le C_RN^{s-r}\left(\|v_N\|_{H^r}+\|w_N\|_{H^r}\right).
		\end{aligned}
	\end{equation}
	This proves \eqref{6.11}.
\end{proof}

\subsection{Fourier--Galerkin approximation}
We first establish the local unique solvability of one time step of the auxiliary Fourier--Galerkin scheme \eqref{3.12}.
\begin{proposition}[Galerkin one-step solvability]\label{proposition6.3}
	For every $R>0$, there exists $\tau_R>0$, independent of $\delta$ and $N$, such that, if
	\begin{equation}
		U_{\delta,N}^n\in X_N,\quad\|U_{\delta,N}^n\|_{H^s}\le R,\quad0<\tau\le\tau_R,
	\end{equation}
	then the $n$-th step of \eqref{3.12} has a unique solution in
	\begin{equation}
		\left\{v_N\in X_N:\|v_N\|_{H^s}\le R+1\right\}.
	\end{equation}
\end{proposition}
\begin{proof}
	Using the equivalent form \eqref{3.13}, the $n$-th Galerkin step can be written as
	\begin{equation}\label{6.17}
		W_N=\mathcal C_{\delta,\tau}U_{\delta,N}^n+\mathcal S_{\delta,\tau}P_N\mathscr G(W_N,U_{\delta,N}^n),
	\end{equation}
	where the right-hand side defines $\mathcal T_n(W_N)$. Since $P_N$ is contractive in $H^s$, estimates \eqref{5.1}, \eqref{5.4}, and \eqref{5.5} give
	\begin{equation}
		\|\mathcal T_n(W_N)\|_{H^s}\le R+\tau K_R,
		\qquad
		\|\mathcal T_n(W_N)-\mathcal T_n(Z_N)\|_{H^s}\le\tau C_R\|W_N-Z_N\|_{H^s}
	\end{equation}
	for $W_N,Z_N$ in the closed $H^s$-ball of radius $R+1$ in $X_N$. Here $K_R$ and $C_R$ are independent of $\delta$ and $N$. Choosing $\tau_RK_R\le1$ and $\tau_RC_R<1$, the contraction mapping theorem gives the asserted unique solution.
\end{proof}

We next compare the Fourier--Galerkin solution $U_{\delta,N}^n$ of \eqref{3.12} with the time-discrete solution $U_\delta^n$ of \eqref{3.6}. We decompose their difference as
\begin{equation}\label{6.18}
	U_\delta^n-U_{\delta,N}^n=\eta^n+\theta^n,
	\quad\eta^n:=U_\delta^n-P_NU_\delta^n,
	\quad\theta^n:=P_NU_\delta^n-U_{\delta,N}^n.
\end{equation}
The next theorem establishes the error of the auxiliary Fourier--Galerkin approximation and a uniform $H^r$-bound for the Galerkin solution, which will be used to control the nonlinear aliasing error in the collocation analysis.

\begin{theorem}[Uniform Galerkin approximation]\label{theorem6.5new}
	Let $\{U_\delta^n\}_{n=0}^{N_T}$ be the time-discrete solution given by Theorem \ref{theorem5.6new}. Then there exist $\tau_0>0$, $N_0\ge1$, $C>0$, and $R>0$, independent of $\delta$, $\tau$, and $N$, such that, for $0<\delta\le\delta_0$, $0<\tau\le\tau_0$, and $N\ge N_0$, the Fourier--Galerkin scheme \eqref{3.12}, with $U_{\delta,N}^0=P_Nu_{\delta,0}$, has a solution sequence satisfying
	\begin{equation}\label{galerkinhsbound}
		\max_{0\le n\le N_T}
		\|U_{\delta,N}^n\|_{H^s}\le R,
	\end{equation}
	\begin{equation}\label{6.24}
		\max_{0\le n\le N_T}\|U_\delta^n-U_{\delta,N}^n\|_{H^s}\le CN^{s-r}
	\end{equation}
	and
	\begin{equation}\label{6.25}
		\max_{0\le n\le N_T}\|U_{\delta,N}^n\|_{H^r}\le C.
	\end{equation}
	Moreover, the solution is unique among all solution sequences satisfying \eqref{galerkinhsbound}.
\end{theorem}
\begin{proof}
	The proof follows the time-level induction used in Theorem \ref{theorem5.6new}. We record the Galerkin-specific projection estimates and the uniform $H^r$ bound.
	
	By \eqref{5.26} in Theorem \ref{theorem5.6new} and the embedding $H^r\hookrightarrow H^s$, there exists $K>0$, independent of $\delta$ and $\tau$, such that
	\begin{equation}
		\max_{0\le n\le N_T}\|U_\delta^n\|_{H^s}\le K.
	\end{equation}
	Set $R=K+3$. For the decomposition \eqref{6.18}, \eqref{6.5} in Lemma \ref{lemma6.1} and \eqref{5.26} in Theorem \ref{theorem5.6new} give
	\begin{equation}\label{6.20}
		\|\eta^j\|_{H^s}\le C_\eta N^{s-r},
		\quad 0\le j\le N_T.
	\end{equation}
	
	Since $U_{\delta,N}^0=P_NU_\delta^0$, we have $\theta^0=0$. Suppose that $U_{\delta,N}^0,\ldots,U_{\delta,N}^n$ have been constructed and
	\begin{equation}
		\max_{0\le j\le n}\|\theta^j\|_{H^s}\le1.
	\end{equation}
	If $N$ is sufficiently large that $C_\eta N^{s-r}\le1$, then \eqref{6.18} and \eqref{6.20} imply
	\begin{equation}
		\|U_{\delta,N}^j\|_{H^s}\le\|U_\delta^j\|_{H^s}+\|\eta^j\|_{H^s}+\|\theta^j\|_{H^s}
		\le K+2,
		\quad 0\le j\le n.
	\end{equation}
	Proposition \ref{proposition6.3} therefore constructs $U_{\delta,N}^{n+1}$ with $\|U_{\delta,N}^{n+1}\|_{H^s}\le R$ for sufficiently small $\tau$.
	
	For every completed step $0\le j\le n$, applying $P_N$ to \eqref{3.6}, using \eqref{6.6} in Lemma \ref{lemma6.1}, and subtracting \eqref{3.13} give
	\begin{equation}\label{6.22}
		\mathrm iD_\tau\theta^j
		=-\mathcal L_\delta\theta^{j+1/2}
		+P_N\!\left[
		\mathscr G(U_\delta^{j+1},U_\delta^j)
		-\mathscr G(U_{\delta,N}^{j+1},U_{\delta,N}^j)
		\right].
	\end{equation}
	The bounds above, the decomposition \eqref{6.18}, the projection estimate \eqref{6.20}, and \eqref{5.1} in Lemma \ref{lemma5.1} yield
	\begin{equation}\label{6.23}
		\begin{aligned}
			&\left\|P_N\!\left[\mathscr G(U_\delta^{j+1},U_\delta^j)-\mathscr G(U_{\delta,N}^{j+1},U_{\delta,N}^j)\right]\right\|_{H^s}\\
			&\qquad\le C_R\left(\|\theta^{j+1}\|_{H^s}+\|\theta^j\|_{H^s}+N^{s-r}\right).
		\end{aligned}
	\end{equation}
	Applying Lemma \ref{lemmaC} to \eqref{6.22} and using \eqref{6.23} give
	\begin{equation}\label{6.21}
		\begin{aligned}
			\|\theta^{j+1}\|_{H^s}^2-\|\theta^j\|_{H^s}^2
			\le{}&C_R\tau\left(\|\theta^{j+1}\|_{H^s}^2+\|\theta^j\|_{H^s}^2\right)\\
			&+C_R\tau N^{2(s-r)}.
		\end{aligned}
	\end{equation}
	Applying Lemma \ref{lemmaD}, as in the proof of Theorem \ref{theorem5.6new}, gives
	\begin{equation}\label{6.26}
		\max_{0\le j\le n+1}\|\theta^j\|_{H^s}
		\le C_\theta N^{s-r}.
	\end{equation}
	
	Choose $\tau_0$ below the threshold in Theorem \ref{theorem5.6new}, the thresholds in Proposition \ref{proposition6.3} corresponding to the input radii $K+2$ and $R$, and the threshold in Lemma \ref{lemmaD}. Choose $N_0$ sufficiently large that $\max\{C_\eta,C_\theta\}N_0^{s-r}\le1$. Then \eqref{6.26} closes the induction, and the Galerkin solution is defined at every time level. Combining \eqref{6.18}, \eqref{6.20}, and \eqref{6.26} proves \eqref{6.24} and \eqref{galerkinhsbound}. The one-step uniqueness in Proposition \ref{proposition6.3}, applied with input radius $R$, proves uniqueness among all solution sequences satisfying \eqref{galerkinhsbound}.
	
	It remains to prove \eqref{6.25}. By \eqref{galerkinhsbound}, the $H^r$-contractivity of $P_N$, and the Sobolev--Moser argument used to derive \eqref{5.27},
	\begin{equation}\label{6.27}
		\left\|P_N\mathscr G(U_{\delta,N}^{n+1},U_{\delta,N}^n)\right\|_{H^r}
		\le C_R\left(\|U_{\delta,N}^{n+1}\|_{H^r}+\|U_{\delta,N}^n\|_{H^r}\right).
	\end{equation}
	Using \eqref{6.17}, the operator estimates \eqref{5.4}--\eqref{5.5} in Lemma \ref{lemmaB}, and \eqref{6.27}, and decreasing $\tau_0$ if necessary, gives
	\begin{equation}\label{6.28}
		\|U_{\delta,N}^{n+1}\|_{H^r}
		\le(1+C_R\tau)\|U_{\delta,N}^n\|_{H^r}.
	\end{equation}
	Iterating \eqref{6.28}, using $U_{\delta,N}^0=P_Nu_{\delta,0}$, the $H^r$-contractivity of $P_N$, and \eqref{4.20} in Theorem \ref{theorem4.6}, proves \eqref{6.25}.
\end{proof}

\subsection{Fourier collocation approximation}
In this subsection, we consider the Fourier collocation scheme \eqref{3.22}, which is the fully discrete method used in the computation. We first prove its one-step solvability and grid conservation laws, and then reconstruct the grid solution as a trigonometric polynomial and compare it with the Galerkin solution obtained in Theorem \ref{theorem6.5new}.
\begin{proposition}[Collocation solvability and conservation]\label{proposition6.6}
	For every $R>0$, there exists $\tau_R>0$, independent of $\delta$ and $N$, such that if $\|U^n\|_{s,h}\le R, 0<\tau\le\tau_R$, then the $n$-th step of \eqref{3.22} has a unique solution in
	\begin{equation}
		B_{R+1}^h:=\left\{V:\|V\|_{s,h}\le R+1\right\}.
	\end{equation}
	Moreover, every solution sequence of \eqref{3.22} satisfies the following mass and energy conservation laws:
	\begin{equation}\label{6.29}
		M_h(U^{n+1})=M_h(U^n), \quad E_{\delta,h}(U^{n+1})=E_{\delta,h}(U^n).
	\end{equation}
\end{proposition}

\begin{proof}
	Define
	\begin{equation}
		\mathcal S_{\delta,\tau}^h=\left(\frac{\mathrm i}{\tau}I+\frac12\mathcal L_{\delta,N}^{c}\right)^{-1},
		\qquad
		\mathcal C_{\delta,\tau}^h=\mathcal S_{\delta,\tau}^h\left(\frac{\mathrm i}{\tau}I-\frac12\mathcal L_{\delta,N}^{c}\right).
	\end{equation}
	These grid operators have the same Fourier symbols on $K_N$ as the operators in Lemma \ref{lemmaB}. Hence, by the proof of Lemma \ref{lemmaB} and the norm identity \eqref{6.3},
	\begin{equation}\label{6.31}
		\|\mathcal S_{\delta,\tau}^hV\|_{s,h}\le\tau\|V\|_{s,h},
		\qquad
		\|\mathcal C_{\delta,\tau}^hV\|_{s,h}=\|V\|_{s,h}.
	\end{equation}
	The $n$-th collocation step is equivalent to $W=\mathcal T_n^h(W)$, where
	\begin{equation}\label{6.32}
		\mathcal T_n^h(W)=\mathcal C_{\delta,\tau}^hU^n+\mathcal S_{\delta,\tau}^h\mathscr G_h(W,U^n).
	\end{equation}
	Estimate \eqref{6.10} in Lemma \ref{lemma6.2} and \eqref{6.31} show that, for $W,Z\in B_{R+1}^h$,
	\begin{equation}
		\|\mathcal T_n^h(W)\|_{s,h}\le R+\tau K_R,
		\qquad
		\|\mathcal T_n^h(W)-\mathcal T_n^h(Z)\|_{s,h}
		\le\tau C_R\|W-Z\|_{s,h}.
	\end{equation}
	Choose $\tau_R>0$ so that $\tau_RK_R\le1$ and $\tau_RC_R<1$. Then $\mathcal T_n^h$ maps $B_{R+1}^h$ into itself and is a contraction. The contraction mapping theorem proves the asserted one-step solvability.
	
	Finally, testing \eqref{3.22} with $U^{n+1/2}$ and taking imaginary parts proves mass conservation. Testing \eqref{3.22} with $D_\tau U^n$ and taking real parts, using the self-adjointness in \eqref{6.4} and the difference-quotient identity \eqref{3.2}, proves energy conservation. Thus \eqref{6.29} follows.
\end{proof}

To estimate the error introduced by Fourier collocation, we compare the collocation solution $U^n$ with the Fourier--Galerkin solution $U_{\delta,N}^n$. We first represent the grid function $U^n$ by its trigonometric interpolant
\begin{equation}\label{6.35}
	V_N^n:=\mathcal I_NU^n\in X_N.
\end{equation}
Applying $\mathcal I_N$ to the collocation scheme \eqref{3.22} and using the identities in \eqref{6.4}, we can obtain
\begin{equation}\label{6.36}
	\mathrm iD_\tau V_N^n=-\mathcal L_\delta V_N^{n+1/2}+I_N\mathscr G(V_N^{n+1},V_N^n).
\end{equation}
Thus, both $U_{\delta,N}^n$ and $V_N^n$ belong to $X_N$. Define their difference by
\begin{equation}\label{6.37}
	\zeta^n:=U_{\delta,N}^n-V_N^n.
\end{equation}

The following theorem establishes solvability at all time levels and uniform convergence of the Fourier collocation scheme.
\begin{theorem}[Uniform collocation approximation]\label{theorem6.8new}
	Let $u_\delta$ denote the solution of the nonlocal NLS model \eqref{2.5}. Then there exist $\tau_0>0$, $N_0\ge1$, $C>0$, and $R>0$, independent of $\delta$, $\tau$, and $N$, such that, for $0<\delta\le\delta_0$, $0<\tau\le\tau_0$, and $N\ge N_0$, the Fourier collocation scheme \eqref{3.22} has a solution sequence $\{U^n\}_{n=0}^{N_T}$ satisfying
	\begin{equation}\label{6.44}
		\max_{0\le n\le N_T}\|U^n\|_{s,h}\le R
	\end{equation}
	and
	\begin{equation}\label{6.45}
		\max_{0\le n\le N_T}
		\left\|
		u_\delta(t_n)-\mathcal I_NU^n
		\right\|_{H^s}
		\le
		C\left(
		\tau^2+N^{s-r}
		\right).
	\end{equation}
	Moreover, the solution is unique among all solution sequences of \eqref{3.22} with the prescribed initial value that satisfy \eqref{6.44}.
\end{theorem}
\begin{proof}
	We follow the time-level induction in Theorems \ref{theorem5.6new} and \ref{theorem6.5new}, and record only the interpolation and aliasing terms specific to collocation. By \eqref{6.25} in Theorem \ref{theorem6.5new} and $H^r\hookrightarrow H^s$, there exists $K>0$, independent of $\delta$, $\tau$, and $N$, such that $\max_{0\le j\le N_T}\|U_{\delta,N}^j\|_{H^s}\le K$. Set $R=K+2$ and $R_*=2K+2$. From the initial conditions, $\zeta^0=P_Nu_{\delta,0}-I_Nu_{\delta,0}$. Hence \eqref{6.5} in Lemma \ref{lemma6.1} and \eqref{4.20} in Theorem \ref{theorem4.6} give
	\begin{equation}\label{6.38}
		\|\zeta^0\|_{H^s}\le CN^{s-r}.
	\end{equation}
	
	Suppose that $U^0,\ldots,U^n$ have been constructed and $\max_{0\le j\le n}\|\zeta^j\|_{H^s}\le1$. By \eqref{6.3} and \eqref{6.37}, $\|U^n\|_{s,h}=\|V_N^n\|_{H^s}\le K+1$, so Proposition \ref{proposition6.6} constructs $U^{n+1}$ with $\|V_N^{n+1}\|_{H^s}\le R$ for sufficiently small $\tau$. Subtracting \eqref{6.36} from \eqref{3.13} gives
	\begin{equation}\label{6.41}
		\begin{aligned}
			\mathrm iD_\tau\zeta^n={}&-\mathcal L_\delta\zeta^{n+1/2}
			+I_N\!\left[\mathscr G(U_{\delta,N}^{n+1},U_{\delta,N}^n)-\mathscr G(V_N^{n+1},V_N^n)\right]\\
			&+(P_N-I_N)\mathscr G(U_{\delta,N}^{n+1},U_{\delta,N}^n).
		\end{aligned}
	\end{equation}
	The bounds above give $\|U_{\delta,N}^{n+1}\|_{H^s}+\|U_{\delta,N}^{n}\|_{H^s}+\|V_N^{n+1}\|_{H^s}+\|V_N^{n}\|_{H^s}\le4K+3<2R_*$. Therefore, applying \eqref{6.10}--\eqref{6.11} in Lemma \ref{lemma6.2} with radius $2R_*$, together with \eqref{6.3}--\eqref{6.4} and \eqref{6.25} in Theorem \ref{theorem6.5new}, gives
	\begin{equation}
		\begin{aligned}
			&\left\|I_N\!\left[\mathscr G(U_{\delta,N}^{n+1},U_{\delta,N}^n)-\mathscr G(V_N^{n+1},V_N^n)\right]\right\|_{H^s}
			\le C_{R_*}\bigl(\|\zeta^{n+1}\|_{H^s}+\|\zeta^n\|_{H^s}\bigr),\\
			&\left\|(P_N-I_N)\mathscr G(U_{\delta,N}^{n+1},U_{\delta,N}^n)\right\|_{H^s}
			\le CN^{s-r}.
		\end{aligned}
	\end{equation}
	Applying Lemma \ref{lemmaC} to \eqref{6.41} yields
	\begin{equation}\label{6.40}
		\|\zeta^{n+1}\|_{H^s}^2-\|\zeta^n\|_{H^s}^2
		\le C_{R_*}\tau\bigl(\|\zeta^{n+1}\|_{H^s}^2+\|\zeta^n\|_{H^s}^2\bigr)+C\tau N^{2(s-r)}.
	\end{equation}
	Applying Lemma \ref{lemmaD} with $a_j=\|\zeta^j\|_{H^s}^2$ and $\varepsilon=N^{s-r}$, using \eqref{6.38} and \eqref{6.40}, gives constants $\tau_\zeta,C_\zeta>0$ independent of $\delta$, $\tau$, and $N$. Choose $\tau_0$ no larger than the threshold in Theorem \ref{theorem6.5new}, the thresholds in Proposition \ref{proposition6.6} corresponding to the input radii $K+1$ and $R$, and $\tau_\zeta$. Choose $N_0$ no smaller than the threshold in Theorem \ref{theorem6.5new} and sufficiently large that $\max\{C,C_\zeta\}N_0^{s-r}\le1$, where $C$ is the constant in \eqref{6.38}. The time-level induction used in Theorems \ref{theorem5.6new} and \ref{theorem6.5new} then gives
	\begin{equation}\label{6.46}
		\max_{0\le n\le N_T}\|\zeta^n\|_{H^s}\le C_\zeta N^{s-r}.
	\end{equation}
	By \eqref{6.3}, \eqref{6.37}, \eqref{6.46}, and the definition of $K$, $\max_n\|U^n\|_{s,h}\le K+1\le R$, which proves \eqref{6.44}. Since $\tau_0$ is below the one-step threshold in Proposition \ref{proposition6.6} for input radius $R$, the same proposition proves uniqueness among all solution sequences satisfying \eqref{6.44}. Finally,
	\begin{equation}
		\|u_\delta(t_n)-\mathcal I_NU^n\|_{H^s}
		\le\|u_\delta(t_n)-U_\delta^n\|_{H^s}+\|U_\delta^n-U_{\delta,N}^n\|_{H^s}+\|\zeta^n\|_{H^s}.
	\end{equation}
	Combining \eqref{5.21} in Theorem \ref{theorem5.6new}, \eqref{6.24} in Theorem \ref{theorem6.5new}, and \eqref{6.46} proves \eqref{6.45}.
\end{proof}
\subsection{Asymptotic compatibility}

\begin{theorem}[Asymptotic compatibility]\label{theorem6.9}
	Assume the hypotheses of Theorem \ref{theorem4.6} hold and $s\in\mathbb N_0$, $s>d/2$, and $r\ge s+6$. Let $\tau_0$ and $N_0$ be the thresholds in Theorem \ref{theorem6.8new}. For $0<\delta\le\delta_0$, $0<\tau\le\tau_0$, and $N\ge N_0$, let $u$ be the solution of the local NLS model \eqref{2.6}, and let $\{U^n\}_{n=0}^{N_T}$ be the solution of the Fourier collocation scheme \eqref{3.22} for the nonlocal NLS model. Then there exists $C>0$, independent of $\delta$, $\tau$, and $N$, such that
	\begin{equation}\label{6.47}
		\max_{0\le n\le N_T}
		\left\|
		u(t_n)-\mathcal I_NU^n
		\right\|_{H^s}
		\le
		C\left(
		\delta^2+\tau^2+N^{s-r}
		\right).
	\end{equation}
\end{theorem}
\begin{proof}
	Let $u_\delta$ be the nonlocal solution in Theorem \ref{theorem4.6}. Then, by the triangle inequality,
	\begin{equation}
		\left\|u(t_n)-\mathcal I_NU^n\right\|_{H^s}
		\le
		\|u(t_n)-u_\delta(t_n)\|_{H^s}+\left\|u_\delta(t_n)-\mathcal I_NU^n\right\|_{H^s}.
	\end{equation}
	Then the proof is complete by the combination of \eqref{4.19} in Theorem \ref{theorem4.6}, the embedding $H^r\hookrightarrow H^s$, and \eqref{6.45} in Theorem \ref{theorem6.8new}.
\end{proof}
\begin{remark}
	Theorem \ref{theorem6.9} holds without any coupling condition among $\delta$, $\tau$, and $N$. Hence the collocation solution converges to the local solution along any sequence for which $\delta\to0$, $\tau\to0$, and $N\to\infty$.
\end{remark}

Finally, we compare the nonlocal grid energy preserved by the collocation scheme with the continuous energy of the local NLS equation.
\begin{proposition}[Convergence of the conserved grid energy]
	Under the assumptions of Theorems \ref{theorem4.6} and \ref{theorem6.8new}, let $\{U^n\}_{n=0}^{N_T}$ be the solution of the Fourier collocation scheme \eqref{3.22}.
	For $0<\delta\le\delta_0$, $0<\tau\le\tau_0$, and $N\ge N_0$, there exists $C>0$, independent of $\delta$, $\tau$, and $N$, such that
	\begin{equation}\label{6.48}
		\max_{0\le n\le N_T}\left|E_{\delta,h}(U^n)-E_0(u(t_n))\right|\le
		C\left(\delta^2+N^{s-r}\right).
	\end{equation}
\end{proposition}
\begin{proof}
	Set $v_N^0=\mathcal I_NU^0=I_Nu_{\delta,0}$. Since $s\in\mathbb N_0$ and $s>d/2$, we have $s\ge1$. By \eqref{6.5} in Lemma \ref{lemma6.1} and \eqref{4.20} in Theorem \ref{theorem4.6},
	\begin{equation}
		\|v_N^0-u_{\delta,0}\|_{H^1}\le CN^{1-r}\le CN^{s-r},
		\qquad
		\|v_N^0\|_{H^1}+\|u_{\delta,0}\|_{H^1}\le C.
	\end{equation}
	Hence \eqref{shuangxiannxing} gives
	\begin{equation}
		\left|a_\delta(v_N^0,v_N^0)-a_\delta(u_{\delta,0},u_{\delta,0})\right|\le CN^{s-r}.
	\end{equation}
	
	Moreover, applying \eqref{A1} in Lemma \ref{lemmaA} with $q=r$ and $w=0$ to $\Phi(z)=F(|z|^2)$, using $F(0)=0$ and \eqref{4.20} in Theorem \ref{theorem4.6}, and then applying \eqref{6.5} in Lemma \ref{lemma6.1} with $q=0$ and $\mu=r$ gives
	\begin{equation}
		\left|\frac1{M^d}\sum_{j\in\mathbb J_M}F(|u_{\delta,0}(x_j)|^2)
		-\frac1{(2\pi)^d}\int_{\mathbb T^d}F(|u_{\delta,0}|^2)\,\mathrm dx\right|
		\le CN^{s-r}.
	\end{equation}
	Consequently, \eqref{3.24} and \eqref{2.8} give
	$|E_{\delta,h}(U^0)-E_\delta(u_{\delta,0})|\le CN^{s-r}$.
	Using the grid conservation law \eqref{6.29} in Proposition \ref{proposition6.6}, the continuous conservation law \eqref{4.11} in Theorem \ref{theorem4.4}, and the energy limit \eqref{4.22} in Corollary \ref{corollary 4.8}, we obtain
	\begin{equation}
		\begin{aligned}
			|E_{\delta,h}(U^n)-E_0(u(t_n))|
			&\le |E_{\delta,h}(U^0)-E_\delta(u_{\delta,0})|
			+|E_\delta(u_\delta(t_n))-E_0(u(t_n))|\\
			&\le C\left(N^{s-r}+\delta^2\right).
		\end{aligned}
	\end{equation}
	Taking the maximum over $0\le n\le N_T$ proves \eqref{6.48}.
\end{proof}

\section{Numerical experiments}\label{sec7}
This section reports numerical results for the conservative Fourier collocation scheme for nonlocal NLS \eqref{3.22}. The nonlinear equation at each time step is solved by the fixed-point iteration \eqref{3.27}. All computations are performed using MATLAB R2025b. Unless stated otherwise, the stopping criterion for the nonlinear iteration is $10^{-13}$.
\subsection{Temporal accuracy}
We test temporal convergence for several horizons in dimensions $d=1,2,3$. We use the cubic nonlinearity
\begin{equation}
	f(|u|^2)u=|u|^2u,
\end{equation}
and the normalized constant kernel
\begin{equation}\label{7.1}
	\rho(z)=\frac{2(d+2)}{|B_1|},\quad z\in B_1.
\end{equation}
Setting $a=\delta|k|$, the Fourier multiplier of $-\mathcal L_\delta$ is given by
\begin{equation}\label{7.2}
	\lambda_\delta(k)
	=
	\begin{cases}
		\displaystyle\frac{6}{\delta^2}\left(1-\frac{\sin a}{a}\right),& d=1,\\
		\displaystyle\frac{8}{\delta^2}\left(1-\frac{2J_1(a)}{a}\right),& d=2,\\
		\displaystyle\frac{10}{\delta^2}\left(1-3\frac{\sin a-a\cos a}{a^3}\right),& d=3,
	\end{cases}
\end{equation}
where $J_1$ denotes the Bessel function of the first kind and the values at $a=0$ are understood by continuity. 

For $A\in\mathbb R$ and $k_0\in\mathbb Z^d$, the cubic nonlocal NLS equation has the plane-wave solution
\begin{equation}\label{7.3}
	u_\delta(x,t)=A\exp\!\left(\mathrm i\bigl(k_0\cdot x-\omega_\delta t\bigr)\right),
	\quad
	\omega_\delta=\lambda_\delta(k_0)+ A^2.
\end{equation}

We take $A=0.7$, $T=0.5$, $(M,k_0)=(17,(3))$, $(13,(2,2))$, and $(9,(1,2,2))$ for $d=1,2,3$, respectively, together with $\delta\in\{0.8,0.4,0.2,0.1\}$ and $\tau=(20\cdot2^j)^{-1}$, $j=0,\ldots,4$. Since $k_0\in K_N$ and the nonlinearity does not generate additional Fourier modes, the spatial error vanishes and this test isolates the temporal error. And the error is measured by
\begin{equation}\label{7.4}
	\mathcal E_{d,\delta}(\tau):=\max_{0\le n\le T/\tau}\left\|u_\delta(t_n)-\mathcal I_NU^n\right\|_{H^2}.
\end{equation}

Figure~ \ref{fig:temporal-convergence} shows second-order temporal convergence for every tested dimension and horizon.
\begin{figure}[htbp]
	\centering
	\includegraphics[width=0.92\textwidth]
	{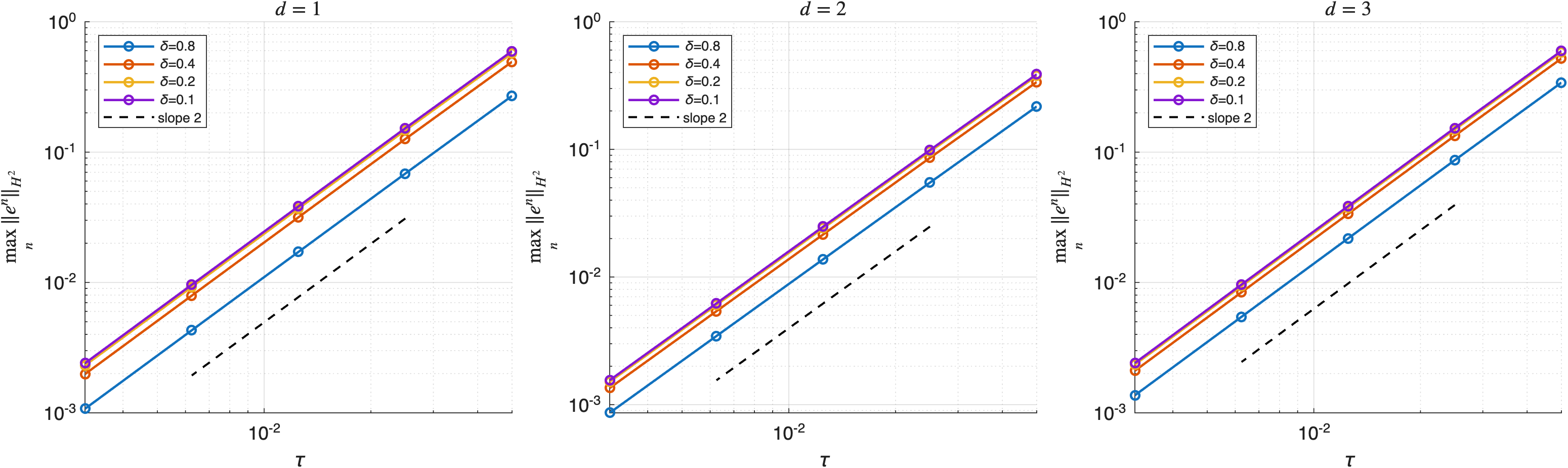}
	\caption{Temporal errors of the Fourier collocation scheme for
		$d=1,2,3$ and four values of $\delta$.}
	\label{fig:temporal-convergence}
\end{figure}

At $\tau=1/320$, the values of $\mathcal E_{d,\delta}(\tau)/\tau^2$ range from $88.276$ to $246.98$ over all tested dimensions and horizons, providing numerical evidence for the horizon-uniform temporal component of \eqref{6.45} in Theorem \ref{theorem6.8new}.

\subsection{Spatial accuracy}
We next test the spatial convergence for the smooth initial value
\begin{equation}\label{7.5}
	u_0(x)=0.8+\frac{0.2}{d}\sum_{\ell=1}^{d}\cos x_\ell+\frac{0.15\mathrm i}{d}\sum_{\ell=1}^{d}\sin(2x_\ell).
\end{equation}

We take $T=0.1$, $\tau=2.5\times10^{-4}$, and $\delta\in\{0.8,0.2,0.05\}$, and compare the numerical solution at $T$ with an over-resolved reference solution computed using the same time step. Figure \ref{fig:spatial-convergence} shows rapid decay of the $H^2$ error in all three dimensions and for all tested horizons, followed by a numerical error floor. This behavior is consistent with Fourier spectral convergence for the smooth solution and with the spatial term in \eqref{6.45} in Theorem \ref{theorem6.8new}.

\begin{figure}[htbp]
	\centering
	\includegraphics[width=0.92\textwidth]
	{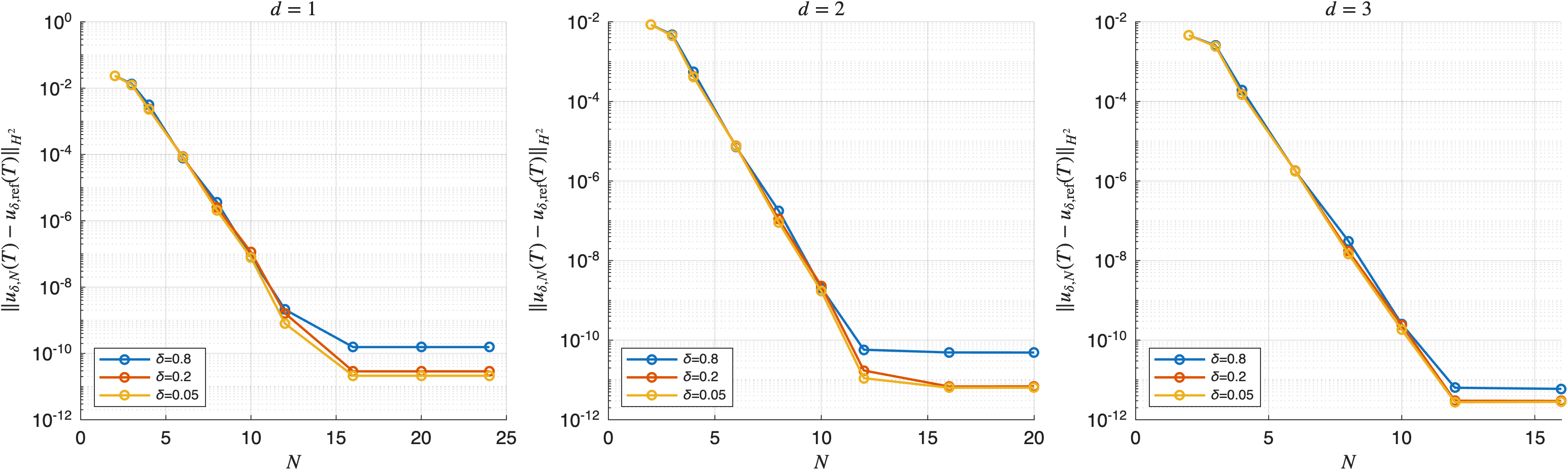}
	\caption{$H^2$ errors of the Fourier collocation scheme for the
		smooth multimode solution in dimensions $d=1,2,3$.}
	\label{fig:spatial-convergence}
\end{figure}

\subsection{Discrete mass and energy conservation}
We next examine the grid mass and energy conservation and the effect of the stopping tolerance in the nonlinear iteration. The initial value is
\begin{equation}\label{7.6}
	u_0(x)=\exp(0.35\cos x)\left(1+0.2\mathrm i\sin(2x)+0.08\cos(3x)\right).
\end{equation}

We use $M=129$, $\delta=0.8$, $T=100$, and $\tau=0.02$ for computation. Moreover, we consider the following cubic nonlinearity and the saturable nonlinearity
\begin{equation}
	f_{\mathrm{cub}}(s)=s,
	\quad
	f_{\mathrm{sat}}(s)=\frac{s}{1+0.7s},
	\quad s\geq0.
\end{equation}
For a nonlinear iteration tolerance $\varepsilon_{\mathrm{it}}$, define the relative defects at $t_n$ by
\begin{equation}
	\mathcal D_M^n=\frac{|M_h(U^n)-M_h(U^0)|}{M_h(U^0)},
	\quad
	\mathcal D_E^n=\frac{|E_{\delta,h}(U^n)-E_{\delta,h}(U^0)|}{1+|E_{\delta,h}(U^0)|}.
\end{equation}

Figure~\ref{fig:grid-conservation} shows the instantaneous relative defects over $5000$ time steps. At $\varepsilon_{\mathrm{it}}=10^{-14}$, the mass and energy defects remain below $10^{-13}$ for the cubic nonlinearity and below $2.1\times10^{-14}$ for the saturable nonlinearity.
\begin{figure}[htbp]
	\centering
	\includegraphics[width=0.92\textwidth]
	{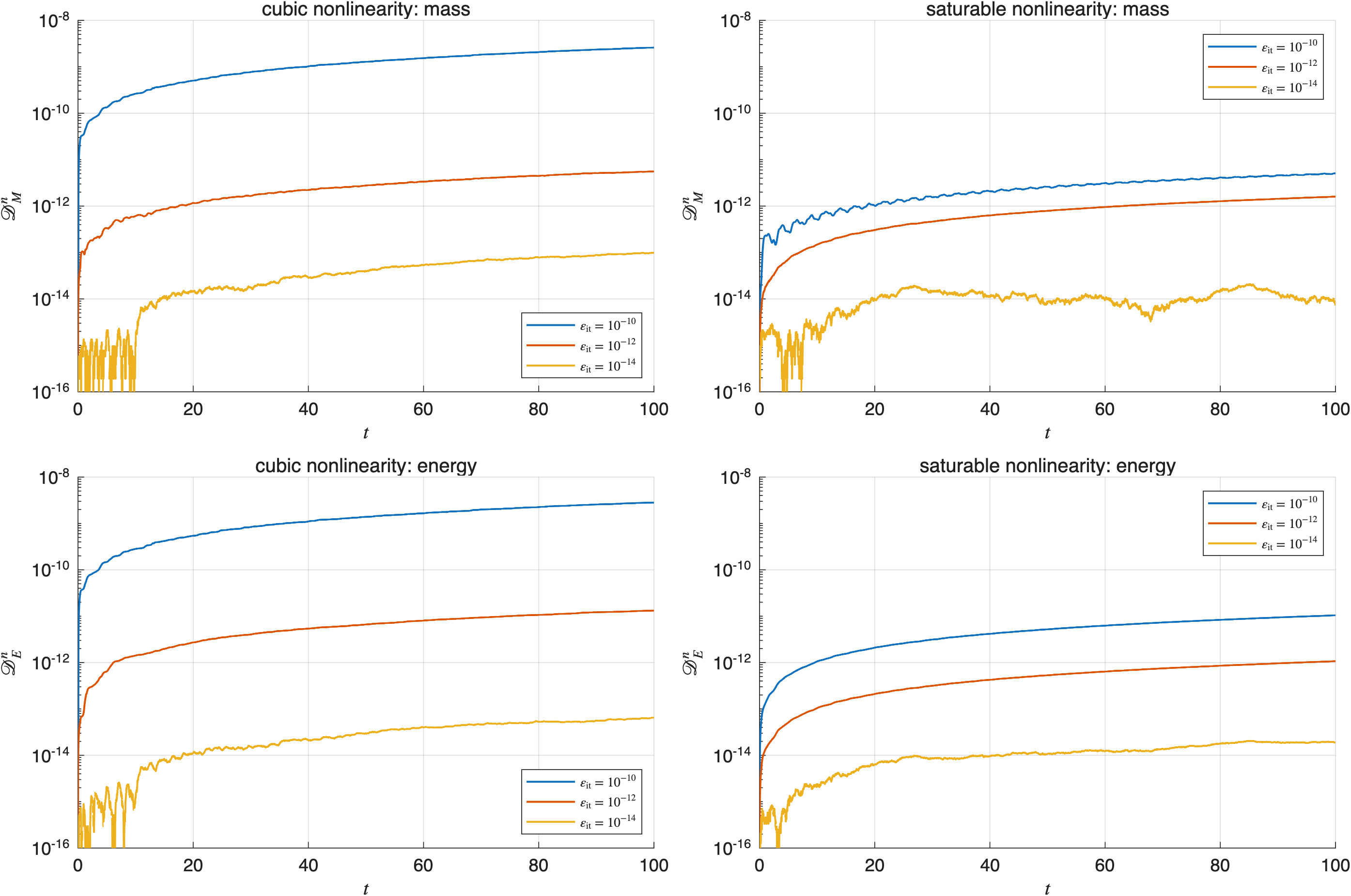}
	\caption{Instantaneous relative defects in the discrete mass and energy for the cubic
		and saturable nonlinearities over $0\le t\le100$.}
	\label{fig:grid-conservation}
\end{figure}

Then, define the maximum defects over the full time interval by $\mathcal D_M=\max_{0\leq n\leq N_T}\mathcal D_M^n$ and $\mathcal D_E=\max_{0\leq n\leq N_T}\mathcal D_E^n$. For each value of $\varepsilon_{\mathrm{it}}$, $R_{\max}$ is the largest discrete $L^2$ residual among all time steps in the cubic and saturable tests. Table \ref{tab:grid-conservation} reports $\mathcal D_M$, $\mathcal D_E$, and $R_{\max}$. As $\varepsilon_{\mathrm{it}}$ decreases, $R_{\max}$ becomes smaller and the mass and energy defects decrease accordingly. At the smallest tolerance, the defects reach the level of double-precision roundoff. These results are consistent with the exact discrete conservation laws in Proposition \ref{proposition6.6} and show that the observed conservation defects are governed by the stopping error of the nonlinear iteration.

\begin{table}[htbp]
	\centering
	\caption{Maximum conservation defects and one-step residuals over
		$0\le t\le100$.}
	\label{tab:grid-conservation}
	\small
	\setlength{\tabcolsep}{5pt}
	\renewcommand{\arraystretch}{0.95}
	\begin{tabular}{c cc cc c}
		\toprule
		& \multicolumn{2}{c}{Cubic}
		& \multicolumn{2}{c}{Saturable}
		& \\
		\cmidrule(lr){2-3}
		\cmidrule(lr){4-5}
		$\varepsilon_{\mathrm{it}}$
		& $\mathcal D_M$ & $\mathcal D_E$
		& $\mathcal D_M$ & $\mathcal D_E$
		& $R_{\max}$ \\
		\midrule
		$10^{-10}$
		& $2.5871\times10^{-9}$ & $2.8169\times10^{-9}$
		& $5.0738\times10^{-12}$ & $1.0435\times10^{-11}$
		& $1.4298\times10^{-10}$ \\
		$10^{-12}$
		& $5.5555\times10^{-12}$ & $1.3202\times10^{-11}$
		& $1.5972\times10^{-12}$ & $1.0697\times10^{-12}$
		& $3.0711\times10^{-12}$ \\
		$10^{-14}$
		& $9.9512\times10^{-14}$ & $6.5313\times10^{-14}$
		& $2.0980\times10^{-14}$ & $2.0593\times10^{-14}$
		& $3.3494\times10^{-14}$ \\
		\bottomrule
	\end{tabular}
\end{table}

\subsection{Nonlocal-to-local convergence}
In this numerical experiment, we examine the local limits of the continuous model, the Fourier collocation solution, and the conserved grid energy in one spatial dimension. We retain the cubic nonlinearity and the normalized constant kernel in \eqref{7.1}. For the continuous model, take $A=0.7$, $k_0=3$, and $T=0.5$. The nonlocal and local plane-wave solutions are
\begin{equation}\label{7.8}
	\begin{aligned}
		u_\delta(x,t)&=A\exp\!\left(\mathrm i\left[k_0x-\left(\lambda_\delta(k_0)+A^2\right)t\right]\right),\\
		u(x,t)&=A\exp\!\left(\mathrm i\left[k_0x-\left(k_0^2+A^2\right)t\right]\right).
	\end{aligned}
\end{equation}
The first error in Table \ref{tab:local-limits} is $\|u_\delta(T)-u(T)\|_{L^2}$, evaluated directly from \eqref{7.8}, thus it contains no temporal or spatial discretization error.

The nonlocal and local collocation schemes are solved on the same grid with $M=129$, $T=0.2$, and $\tau=5\times10^{-4}$ with the same initial function \eqref{7.6}. For the local computation, we solve the counterpart of \eqref{3.22} in which the nonlocal operator $-\mathcal L_{\delta,N}^{c}$ is replaced by the local Fourier collocation operator. Equivalently, its Fourier multiplier $\lambda_\delta(k)$ is replaced by $k^2$. We measure the final time error by
\begin{equation}
	\mathcal E_{\mathrm{col}}(\delta)=\left\|\mathcal I_NU_\delta^{N_T}-\mathcal I_NU_0^{N_T}\right\|_{H^2}.
\end{equation}

Moreover, since the local energy is conserved, $E_0(u(t_n))=E_0(u_0)$. The convergence of the nonlocal grid energy to the local continuous energy is therefore measured by
\begin{equation}
	\mathcal E_E(\delta)=\max_{0\leq n\leq N_T}\left|E_{\delta,h}(U_\delta^n)-E_0(u_0)\right|.
\end{equation}

Figure \ref{fig:local-limits} shows that the plane-wave, collocation, and energy errors are parallel to the reference lines of slope two. Correspondingly, the observed orders in Table \ref{tab:local-limits} approach two as $\delta$ decreases. The plane-wave result agrees with the continuous local-limit estimate in Theorem \ref{theorem4.6}; the collocation and energy results demonstrate the discrete local limit and the energy estimate \eqref{6.48}, respectively.
\begin{figure}[htbp]
	\centering
	\includegraphics[width=0.98\textwidth]
	{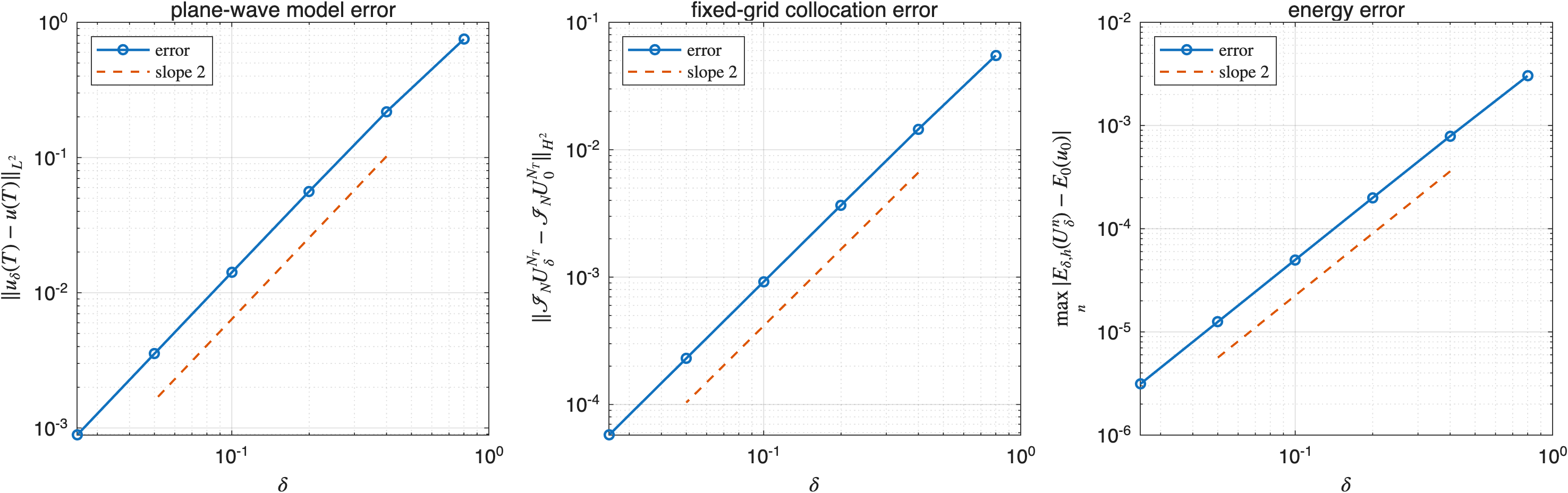}
	\caption{Nonlocal-to-local convergence of the plane-wave solution,
		the Fourier collocation solution, and the conserved grid energy.}
	\label{fig:local-limits}
\end{figure}

\begin{table}[htbp]
	\centering
	\caption{Errors and observed orders in the nonlocal-to-local limits.}
	\label{tab:local-limits}
	\small
	\setlength{\tabcolsep}{4pt}
	\renewcommand{\arraystretch}{0.95}
	\begin{tabular}{c cc cc cc}
		\toprule
		& \multicolumn{2}{c}{Plane wave}
		& \multicolumn{2}{c}{Collocation}
		& \multicolumn{2}{c}{Energy} \\
		\cmidrule(lr){2-3}\cmidrule(lr){4-5}\cmidrule(lr){6-7}
		$\delta$ & Error & Order & Error & Order & Error & Order \\
		\midrule
		$0.8$   & $7.5062\times10^{-1}$ & --
		& $5.4982\times10^{-2}$ & --
		& $3.0257\times10^{-3}$ & -- \\
		$0.4$   & $2.1828\times10^{-1}$ & $1.7819$
		& $1.4501\times10^{-2}$ & $1.9228$
		& $7.8882\times10^{-4}$ & $1.9395$ \\
		$0.2$   & $5.6201\times10^{-2}$ & $1.9575$
		& $3.6694\times10^{-3}$ & $1.9825$
		& $1.9930\times10^{-4}$ & $1.9848$ \\
		$0.1$   & $1.4144\times10^{-2}$ & $1.9904$
		& $9.1988\times10^{-4}$ & $1.9960$
		& $4.9956\times10^{-5}$ & $1.9962$ \\
		$0.05$  & $3.5418\times10^{-3}$ & $1.9977$
		& $2.3012\times10^{-4}$ & $1.9990$
		& $1.2497\times10^{-5}$ & $1.9990$ \\
		$0.025$ & $8.8582\times10^{-4}$ & $1.9994$
		& $5.7541\times10^{-5}$ & $1.9998$
		& $3.1248\times10^{-6}$ & $1.9998$ \\
		\bottomrule
	\end{tabular}
\end{table}

\subsection{Asymptotic compatibility}
Now we examine the simultaneous limits $\delta\to0$, $\tau\to0$, and $N\to\infty$ without imposing a coupling condition. We use the cubic nonlinearity and the normalized constant kernel in \eqref{7.1}. All nonlocal collocation solutions and the local reference solution are initialized with \eqref{7.6}, and the local reference solution is computed on a grid with $M_{\mathrm{ref}}=257$ and time step $\tau_{\mathrm{ref}}=7.8125\times10^{-5}$.  To verify its accuracy, we repeat the reference computation with $513$ grid points and time step $\tau_{\mathrm{ref}}/2$. The two reference solutions differ by $2.62\times10^{-7}$ in $H^2$ at $T$, which is negligible compared with the errors reported below.

To vary the three parameters independently, we take $\delta=0.4\cdot2^{-j}$, $\tau=0.02\cdot2^{-\ell}$, $j,\ell=0,\ldots,3$, and $N\in\{3,4,6,8,12,16,32,64\}$. We solve all $4\times4\times8=128$ combinations in the resulting Cartesian product, which covers $0.15\le N\delta\le25.6$. For each parameter triple, define
\begin{equation}
	\mathcal E_{\mathrm{AC}}(\delta,\tau,N)=\max_{0\le n\le N_T}\left\|u_{\mathrm{ref}}(t_n)-\mathcal I_NU^n\right\|_{H^2}.
\end{equation}
To compare $\mathcal E_{\mathrm{AC}}$ with the estimate in Theorem \ref{theorem6.9}, we introduce the corresponding theoretical error scale. Since the error is measured in $H^2$, we set $s=2$ and $r=8$. We then define
\begin{equation}
	\eta_{\mathrm{th}}(\delta,\tau,N)=\delta^2+\tau^2+N^{s-r}=\delta^2+\tau^2+N^{-6}.
\end{equation}

Then to examine convergence under simultaneous refinement, we select three representative refinement sequences from the $128$ computed cases. In each sequence, both $\delta$ and $\tau$ are divided by two at each step. The Fourier cutoffs are
\begin{equation}
	N^{(1)}=(4,6,8,12),
	\quad
	N^{(2)}=(4,8,16,32),
	\quad
	N^{(3)}=(4,12,32,64).
\end{equation}
For these three sequences, $N\delta$ decreases from $1.6$ to $0.6$, remains equal to $1.6$, and increases from $1.6$ to $3.2$, respectively.

The left panel of Figure \ref{fig:uncoupled-ac} shows second-order convergence along all three paths as the parameters are refined. The right panel shows $\mathcal E_{\mathrm{AC}}/\eta_{\mathrm{th}}$ for these paths. All ratios lie between $1.175$ and $1.374$, so the error remains comparable to $\delta^2+\tau^2+N^{-6}$ for each refinement pattern. The three paths therefore verify Theorem \ref{theorem6.9}, which gives an error bound with a constant independent of $\delta$, $\tau$, and $N$ and imposes no coupling condition on these parameters.

\begin{figure}[htbp]
	\centering
	\includegraphics[width=0.96\textwidth]
	{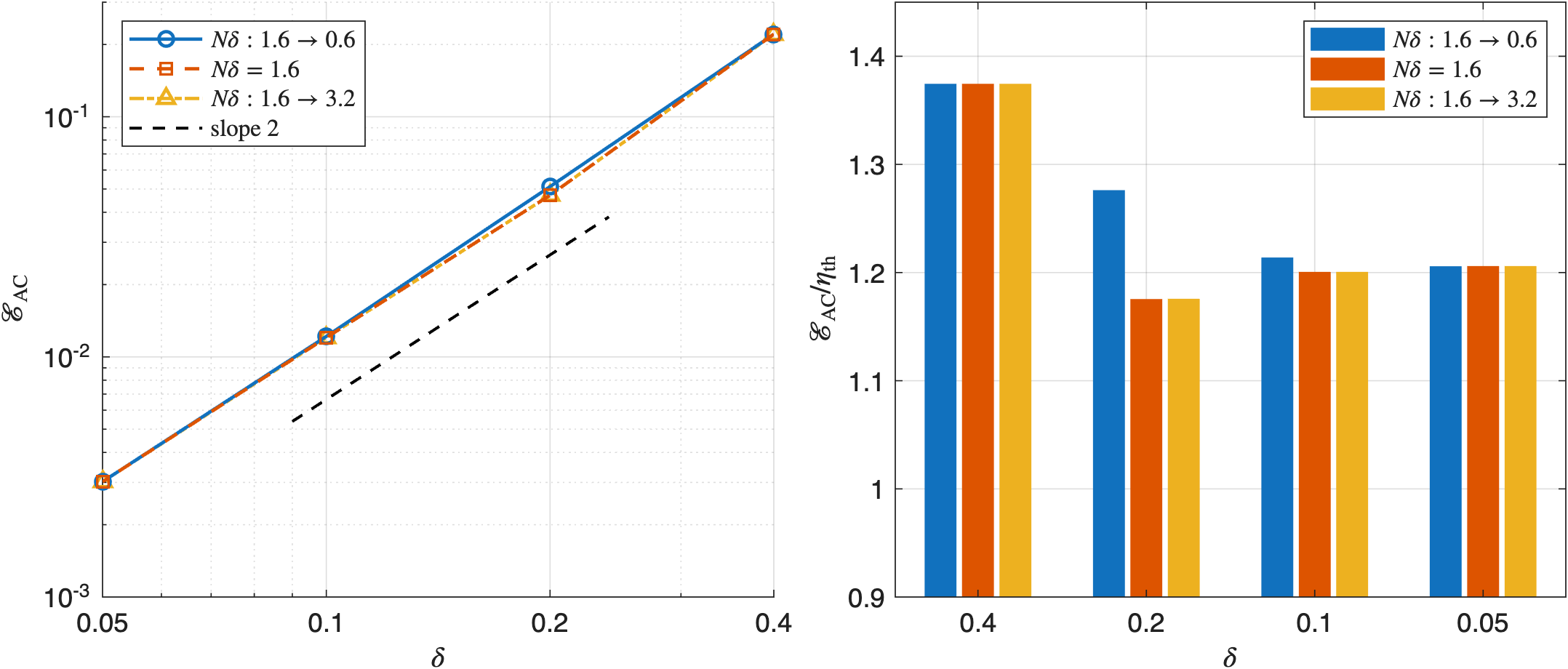}
	\caption{Asymptotic-compatibility test along three independent
		refinement paths. Left: the $H^2$ error $\mathcal E_{\mathrm{AC}}$
		and a reference line of slope two. Right: the normalized error
		$\mathcal E_{\mathrm{AC}}/\eta_{\mathrm{th}}$.}
	\label{fig:uncoupled-ac}
\end{figure}
\subsection{Nonlocal dispersion and wave-packet dynamics}
In this subsection, we examine how the horizon and the kernel modify the dispersion relation, and then study the resulting change in wave-packet propagation.

We investigate how the choice of kernel affects the dispersion relation and group velocity. We consider
\begin{equation}
	\rho_{\mathrm c}(z)=3,
	\quad
	\rho_{\mathrm l}(z)=12(1-|z|),
	\quad
	\rho_{\mathrm s}(z)=\frac{105}{8}(1-z^2)^2,
	\quad |z|\le1.
\end{equation}

All three kernels are even, nonnegative, and satisfy the second-moment normalization \eqref{K2}. And one can easily verify that $m_0(\rho_{\mathrm c})=6$, $m_0(\rho_{\mathrm l})=12$, and $m_0(\rho_{\mathrm s})=14$. To examine the low- and high-frequency conclusions of Proposition \ref{proposition4.9}, we set $a=\delta|\xi|$ and compute
\begin{equation}
	\delta^2\lambda_\delta(\xi)=\int_{-1}^{1}\rho(z)\bigl(1-\cos(az)\bigr)\,\mathrm dz,
	\quad
	\delta|v_{g,\delta}(\xi)|=\left|\int_{-1}^{1}\rho(z)z\sin(az)\,\mathrm dz\right|.
\end{equation}

The left panel of Figure \ref{fig:dispersion-properties} shows the scaled Fourier multipliers. For small $a$, all three curves approach the local reference $a^2$, in agreement with the first estimate in \eqref{4.25}; for large $a$, they approach the corresponding values of $m_0(\rho)$, confirming the first limit in \eqref{4.26}. The right panel shows the scaled group velocities. Their agreement with $2a$ for small $a$ is consistent with the second estimate in \eqref{4.25}, while their decay to zero for large $a$ confirms the second limit in \eqref{4.26}. Thus, $m_0(\rho)$ determines the high-frequency limit of the scaled multiplier, while the shape of $\rho$ affects the transition from the local to the high-frequency regime. The horizon determines the corresponding range of physical wave numbers through $a=\delta|\xi|$.

\begin{figure}[htbp]
	\centering
	\includegraphics[width=0.99\textwidth]
	{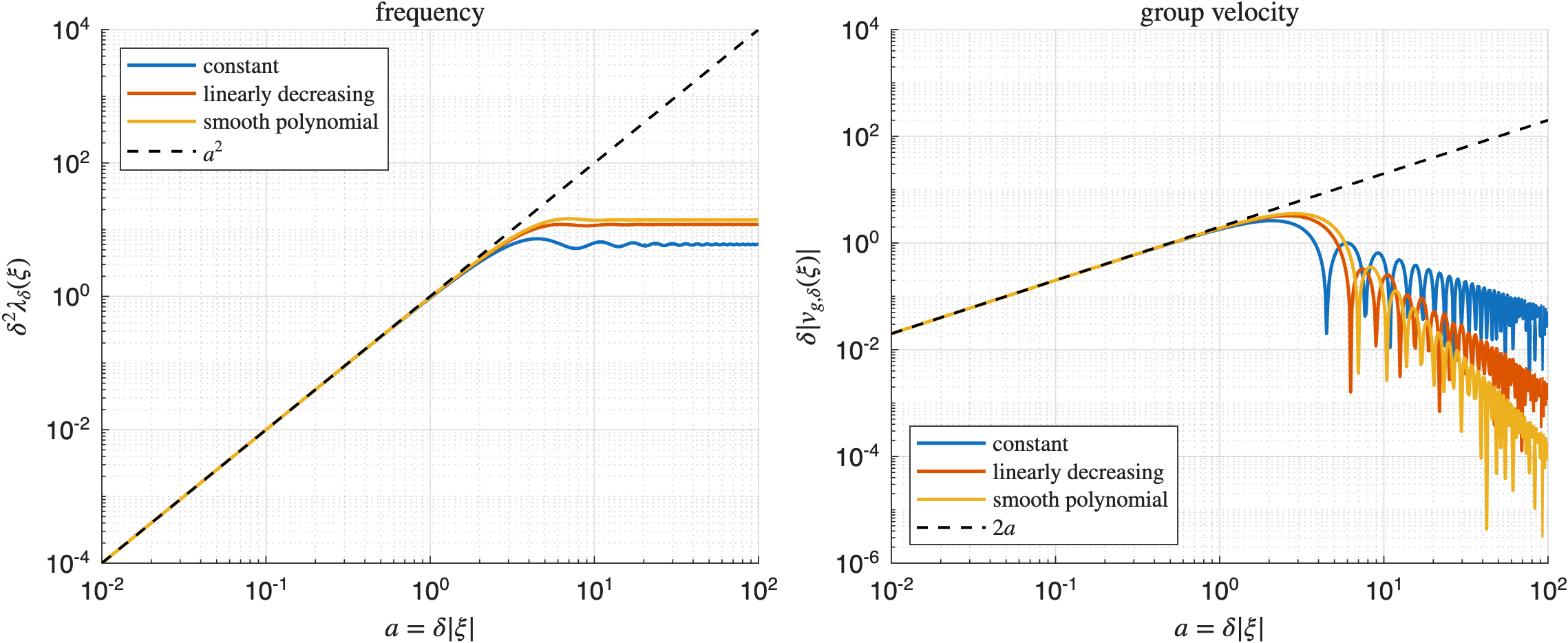}
	\caption{Dispersion for the three kernels considered above.
		Left: scaled Fourier multipliers and the local reference $a^2$. Right: scaled
		group velocities and the local reference $2a$.}
	\label{fig:dispersion-properties}
\end{figure}

We next solve the focusing cubic equation, $f(s)=-s$, on $\mathbb T^2$ with the normalized constant kernel \eqref{7.1}. For $b=(b_1,b_2)$, set
\begin{equation}
	g_b(x)=\exp\!\left\{\kappa\left[\cos(x_1-b_1)+\cos(x_2-b_2)-2\right]\right\},
	\quad
	u_0^{\mathrm{dyn}}(x)=A g_{-x_c}(x)e^{\mathrm i k_0\cdot x}+A g_{x_c}(x)e^{-\mathrm i k_0\cdot x}.
\end{equation}
We take $A=1$, $\kappa=8$, $x_c=(1.8,1.2)$, $k_0=(3,2)$, $M=129$ in each direction, and $\tau=10^{-3}$. Figure \ref{fig:wave-packet-dynamics} compares the density $|U(x_1,x_2,t)|^2$ at $t=0.45$ for the local solution and the nonlocal solutions with $\delta=0.2,0.3,\ldots,1.0$. The same initial value and numerical parameters are used in all computations.

For small horizons, the interaction pattern remains close to that of the local solution. As the horizon increases, the packets travel a shorter distance by $t=0.45$ and their collision is delayed, consistently with the reduction in the nonlocal group velocity. 
\begin{figure}[htbp]
	\centering
	\includegraphics[width=0.99\textwidth]
	{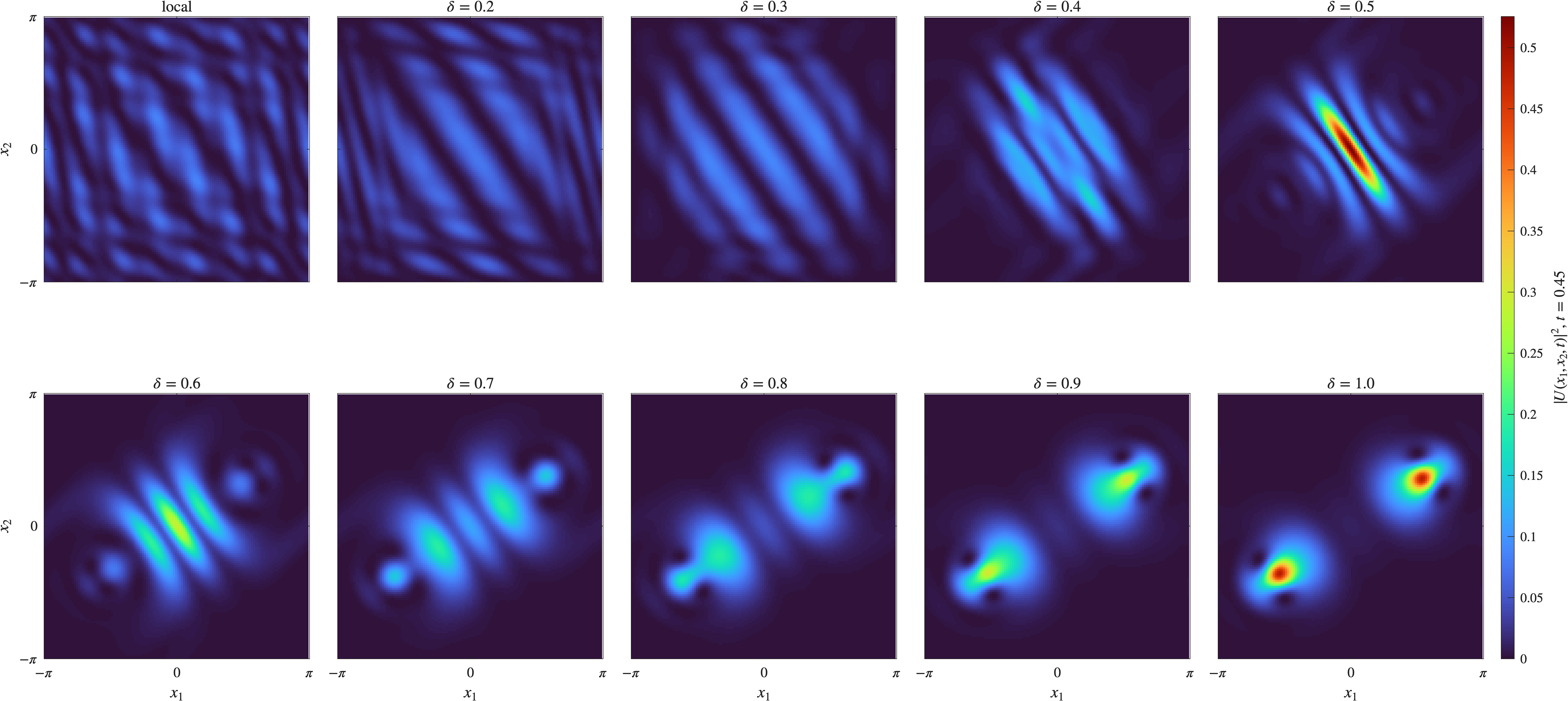}
	\caption{Density $|U(x_1,x_2,t)|^2$ at $t=0.45$ for the local and
		nonlocal focusing cubic NLS equations.}
	\label{fig:wave-packet-dynamics}
\end{figure}

Together, the dispersion tests and wave packet simulations show that the nonlocal model recovers the behavior of the local model for long waves, while the kernel and horizon produce distinct dynamics at shorter wavelengths.

\section{Conclusion}\label{sec8}
In this paper, we studied a finite-horizon nonlocal NLS model on periodic domains. We established global well-posedness, conservation of mass and nonlocal energy, an $O(\delta^2)$ local limit, and the long- and short-wave dispersion properties. We constructed an asymptotically compatible Crank--Nicolson Fourier collocation method that preserves the grid mass and the discrete counterpart of the nonlocal energy. Its horizon-uniform $H^s$-error is $O(\tau^2+N^{s-r})$ relative to the nonlocal solution and $O(\delta^2+\tau^2+N^{s-r})$ relative to the local solution, without any coupling condition among $\delta$, $\tau$, and $N$. Numerical experiments confirm these results, and the two-dimensional wave-packet test illustrates finite-horizon dispersive effects.

\section*{Declaration of competing interest}
The authors declare that they have no known competing financial interests or personal relationships that could have appeared to influence the work reported in this paper.

\section*{Acknowledgments}
The authors gratefully acknowledge the financial support received from the National Natural Science Foundation of China, PR China (No. 12501565, No. 12571440).

\section*{Data availability}
Data will be made available on request.

\bibliographystyle{plain} 
\bibliography{sample}

@article{2012Analysis,
	title={Analysis and Approximation of Nonlocal Diffusion Problems with Volume Constraints},
	author={ Du, Qiang  and  Gunzburger, Max  and  Lehoucq, R. B.  and  Zhou, Kun },
	journal={SIAM Review},
	volume={54},
	number={4},
	pages={667-696},
	year={2012},
}

@article{2013A,
	title={A nonlocal vector calculus, nonlocal volume-constrained problems, and nonlocal balance laws},
	author={ Du, Qiang  and  Gunzburger, Max  and  Lehoucq, R. B.  and  Zhou, Kun },
	journal={Mathematical Models \& Methods in Applied Sciences},
	volume={23},
	number={03},
	pages={493-540},
	year={2013},
}

@article{2017Fast,
	title={Fast and accurate implementation of Fourier spectral approximations of nonlocal diffusion operators and its applications},
	author={ Du, Qiang  and  Yang, Jiang },
	journal={Journal of Computational Physics},
	volume={332},
	pages={118-134},
	year={2017},
}

@book{taylor2013partial,
	title={Partial Differential Equations {III}: Nonlinear Equations},
	author={Taylor, Michael},
	volume={117},
	year={2013},
	publisher={Springer Science \& Business Media}
}

@book{deimling2006ordinary,
	title={Ordinary differential equations in Banach spaces},
	author = {Deimling, Klaus},
	year={2006},
	publisher={Springer}
}

@article{lipton2025energy,
	title={Energy balance and damage for dynamic fast crack growth from a nonlocal formulation},
	author={Lipton, Robert P and Bhattacharya, Debdeep},
	journal={Journal of Elasticity},
	volume={157},
	number={1},
	pages={5},
	year={2025},
	publisher={Springer}
}

@article{shen2021peridynamic,
	title={Peridynamic modeling with energy-based surface correction for fracture simulation of random porous materials},
	author={Shen, Shangkun and Yang, Zihao and Han, Fei and Cui, Junzhi and Zhang, Jieqiong},
	journal={Theoretical and Applied Fracture Mechanics},
	volume={114},
	pages={102987},
	year={2021},
	publisher={Elsevier}
}

@article{javili2019peridynamics,
	title={Peridynamics review},
	author={Javili, Ali and Morasata, Rico and Oterkus, Erkan and Oterkus, Selda},
	journal={Mathematics and Mechanics of Solids},
	volume={24},
	number={11},
	pages={3714--3739},
	year={2019},
	publisher={SAGE Publications Sage UK: London, England}
}

@article{huang2024asymptotic,
	title={Asymptotic compatibility of a class of numerical schemes for a nonlocal traffic flow model},
	author={Huang, Kuang and Du, Qiang},
	journal={SIAM Journal on Numerical Analysis},
	volume={62},
	number={3},
	pages={1119--1144},
	year={2024},
	publisher={SIAM}
}

@article{huang2022stability,
	title={Stability of a nonlocal traffic flow model for connected vehicles},
	author={Huang, Kuang and Du, Qiang},
	journal={SIAM Journal on Applied Mathematics},
	volume={82},
	number={1},
	pages={221--243},
	year={2022},
	publisher={SIAM}
}

@article{pal2025nonlocal,
	title={Nonlocal models in biology and life sciences: Sources, developments, and applications},
	author={Pal, Swadesh and Melnik, Roderick},
	journal={Physics of life reviews},
	volume={53},
	pages={24--75},
	year={2025},
	publisher={Elsevier}
}

@article{tao2018nonlocal,
	title={Nonlocal neural networks, nonlocal diffusion and nonlocal modeling},
	author={Tao, Yunzhe and Sun, Qi and Du, Qiang and Liu, Wei},
	journal={Advances in Neural Information Processing Systems},
	volume={31},
	year={2018}
}

@article{you2022nonlocal,
	title={Nonlocal kernel network (NKN): A stable and resolution-independent deep neural network},
	author={You, Huaiqian and Yu, Yue and D'Elia, Marta and Gao, Tian and Silling, Stewart},
	journal={Journal of Computational Physics},
	volume={469},
	pages={111536},
	year={2022},
	publisher={Elsevier}
}

@article{tian2014asymptotically,
	title={Asymptotically compatible schemes and applications to robust discretization of nonlocal models},
	author={Tian, Xiaochuan and Du, Qiang},
	journal={SIAM Journal on Numerical Analysis},
	volume={52},
	number={4},
	pages={1641--1665},
	year={2014},
	publisher={SIAM}
}

@article{tian2020asymptotically,
	title={Asymptotically compatible schemes for robust discretization of parametrized problems with applications to nonlocal models},
	author={Tian, Xiaochuan and Du, Qiang},
	journal={SIAM Review},
	volume={62},
	number={1},
	pages={199--227},
	year={2020},
	publisher={SIAM}
}

@article{sanz1984methods,
	title={Methods for the numerical solution of the nonlinear Schr{\"o}dinger equation},
	author={Sanz-Serna, JM},
	journal={mathematics of computation},
	volume={43},
	number={167},
	pages={21--27},
	year={1984}
}

@article{henning2017crank,
	title={{Crank--Nicolson Galerkin} approximations to nonlinear Schr{\"o}dinger equations with rough potentials},
	author={Henning, Patrick and Peterseim, Daniel},
	journal={Mathematical Models and Methods in Applied Sciences},
	volume={27},
	number={11},
	pages={2147--2184},
	year={2017},
	publisher={World Scientific}
}

@article{wang2018structure,
	title={Structure-preserving numerical methods for the fractional Schr{\"o}dinger equation},
	author={Wang, Pengde and Huang, Chengming},
	journal={Applied Numerical Mathematics},
	volume={129},
	pages={137--158},
	year={2018},
	publisher={Elsevier}
}

@article{ding2024construction,
	title={Construction and analysis of structure-preserving numerical algorithm for two-dimensional damped nonlinear space fractional Schr{\"o}dinger equation},
	author={Ding, Hengfei and Qu, Haidong and Yi, Qian},
	journal={Journal of Scientific Computing},
	volume={99},
	number={3},
	pages={60},
	year={2024},
	publisher={Springer}
}

@article{zhang2025high,
	title={High-order mass-and energy-conserving methods for the coupled nonlinear Schr{\"o}dinger equation},
	author={Zhang, Pingrui and Xia, Li and Zhang, Hui and Jiang, Xiaoyun},
	journal={Communications in Nonlinear Science and Numerical Simulation},
	volume={149},
	pages={108944},
	year={2025},
	publisher={Elsevier}
}

@article{alali2020fourier,
	title={Fourier spectral methods for nonlocal models},
	author={Alali, Bacim and Albin, Nathan},
	journal={Journal of Peridynamics and Nonlocal Modeling},
	volume={2},
	number={3},
	pages={317--335},
	year={2020},
	publisher={Springer}
}

@article{yan2020numerical,
	title={Numerical computations of nonlocal Schr{\"o}dinger equations on the real line},
	author={Yan, Yonggui and Zhang, Jiwei and Zheng, Chunxiong},
	journal={Communications on Applied Mathematics and Computation},
	volume={2},
	number={2},
	pages={241--260},
	year={2020},
	publisher={Springer}
}

@article{wang2022stability,
	title={Stability and error analysis for a second-order approximation of 1D nonlocal Schr{\"o}dinger equation under DtN-type boundary conditions},
	author={Wang, Jihong and Zhang, Jiwei and Zheng, Chunxiong},
	journal={Mathematics of Computation},
	volume={91},
	number={334},
	pages={761--783},
	year={2022}
}

@article{laskin2000fractional,
	title={Fractional quantum mechanics and L{\'e}vy path integrals},
	author={Laskin, Nikolai},
	journal={Physics Letters A},
	volume={268},
	number={4-6},
	pages={298--305},
	year={2000},
	publisher={Elsevier}
}

@article{d2013fractional,
	title={The fractional Laplacian operator on bounded domains as a special case of the nonlocal diffusion operator},
	author={D’Elia, Marta and Gunzburger, Max},
	journal={Computers \& Mathematics with Applications},
	volume={66},
	number={7},
	pages={1245--1260},
	year={2013},
	publisher={Elsevier}
}

@article{kirkpatrick2013continuum,
	title={On the continuum limit for discrete NLS with long-range lattice interactions},
	author={Kirkpatrick, Kay and Lenzmann, Enno and Staffilani, Gigliola},
	journal={Communications in mathematical physics},
	volume={317},
	number={3},
	pages={563--591},
	year={2013},
	publisher={Springer}
}

\end{document}